\documentclass[11pt]{article}
\usepackage{amsmath}
\usepackage{amssymb}

\newtheorem{theorem}{Theorem}

\newtheorem{corollary}[theorem]{Corollary}

\newtheorem{definition}[theorem]{Definition}

\newtheorem{lemma}[theorem]{Lemma}

\newtheorem{proposition}[theorem]{Proposition}
\newtheorem{remark}[theorem]{Remark}

\newenvironment{proof}[1][Proof]{\noindent\textbf{#1.} }{\ \rule{0.5em}{0.5em}}

\let\geq\geqslant
\let\leq\leqslant
\input{tcilatex}
\begin{document}

\title{On $L_{p}$- theory for stochastic parabolic integro-differential
equations}
\author{R. Mikulevicius and H. Pragarauskas \\
%EndAName
University of Southern California, Los Angeles\\
Institute of Mathematics and Informatics, Vilnius }
\maketitle

\begin{abstract}
The existence and uniqueness in fractional Sobolev spaces of the Cauchy
problem to a stochastic parabolic integro-differential equation is
investigated. A model problem with coefficients independent of space
variable is considered. The equation arises in a filtering problem with a
jump signal and jump observation process.
\end{abstract}

\section{ \protect\bigskip Introduction}

Let $(\Omega ,\mathcal{F},\mathbf{P})$ be a complete probability space with
a filtration of $\sigma $-algebras $\mathbb{F}=(\mathcal{F}_{t},t\geq 0)$
satisfying the usual conditions. Let $\mathcal{R}(\mathbb{F})$ be the
progressive $\sigma $-algebra on $[0,\infty )\times \Omega $. Let $(U,%
\mathcal{U},\Pi )$ be a measurable space with a $\sigma $-finite measure $%
\Pi ,\mathbf{R}_{0}^{d}=\mathbf{R}^{d}\backslash \{0\}$. Let $p^{(\alpha
)}(dt,dy),\ \alpha \in (0,2)$, and $\nu (dt,d\upsilon )$ be $\mathbb{F}$%
-adapted point measures on $([0,\infty )\times \mathbf{R}_{0}^{d},\mathcal{B}%
([0,\infty ))\otimes \mathcal{B}(\mathbf{R}_{0}^{d}))$ and $([0,\infty
)\times U,\mathcal{B}([0,\infty ))\otimes \mathcal{U})$ with compensators $%
l^{(\alpha )}(t,y)dydt/|y|^{d+\alpha }$ and $\Pi (d\upsilon )dt$. We assume
that the measures $\nu $ and $p^{(\alpha )},\ \alpha \in (0,2)$, have no
common jumps.

Let $E=[0,T]\times \mathbf{R}^{d}$. For fixed $\alpha \in (0,2]$, we
consider the linear stochastic integro-differential parabolic equation 
\begin{eqnarray}
du(t,x) &=&\bigl(A^{(\alpha )}u(t,x)-\lambda u(t,x)+f(t,x)\bigr)dt
\label{intr1} \\
&&+\int_{\mathbf{R}_{0}^{d}}[u(t-,x+y)-u(t-,x)+g(t,x,y)]q^{(\alpha
)}(dt,dy)1_{\alpha \in (0,2)}  \notag \\
&&+[1_{\alpha =2}\sigma ^{i}(t)\partial
_{i}u(t,x)+h(t,x)]dW_{t}\,+\int_{U}\Phi (t,x,\upsilon )\eta (dt,d\upsilon )%
\text{\quad in }E,  \notag \\
u(0,x) &=&u_{0}(x)\text{ in }\mathbf{R}^{d},  \notag
\end{eqnarray}%
where $\lambda \geq 0,$ $W_{t}$ is a cylindrical $\mathbb{F}$-adapted Wiener
process in a separable Hilbert space $Y$ and $q^{(\alpha )},\ \alpha \in
(0,2),\ \eta $ are martingale measures defined by 
\begin{equation*}
q^{(\alpha )}(dt,dy)=p^{(\alpha )}(dt,dy)-l^{(\alpha )}(t,y)\frac{dydt}{%
|y|^{d+\alpha }}
\end{equation*}%
and 
\begin{equation*}
\eta (dt,d\upsilon )=\nu (dt,d\upsilon )-\Pi (d\upsilon )dt.
\end{equation*}

The input functions $u_{0},f,g,\Phi ,h$ satisfy the following measurability
assumptions: $u_{0}$ is $\mathcal{F}_{0}\otimes \mathcal{B}(\mathbf{R}^{d})$%
-measurable, $g$ is $\mathcal{R}(\mathbb{F)}\otimes \mathcal{B}(\mathbf{R}%
^{d})\otimes \mathcal{B}(\mathbf{R}_{0}^{d})$-measurable, $f$ is $\mathcal{R}%
(\mathbb{F})\otimes \mathcal{B}(\mathbf{R}^{d})$-measurable, $\Phi $ is $%
\mathcal{R}(\mathbb{F)\otimes }\mathcal{B}(\mathbf{R}^{d})\otimes \mathcal{U}
$-measurable and $h$ is $Y$-valued and $\mathcal{R}(\mathbb{F})\otimes 
\mathcal{B}(\mathbf{R}^{d})$-measurable. The operator $A^{(\alpha )}$ is
defined as 
\begin{eqnarray}
A^{(\alpha )}u(t,x) &=&\int_{\mathbf{R}_{0}^{d}}\nabla _{y}^{\alpha
}u(t,x)m^{(\alpha )}(t,y)\frac{dy}{|y|^{d+\alpha }}1_{\alpha \in (0,2)}
\label{nf1} \\
&&+\bigl(b(t),\nabla u(t,x)\bigr)1_{\alpha =1}+\frac{1}{2}B^{ij}(t)\partial
_{ij}^{2}u(t,x)1_{\alpha =2},  \notag
\end{eqnarray}%
where 
\begin{equation*}
\nabla _{y}^{\alpha }u(t,x)=u(t,x+y)-u(t,x)-(\nabla u(t,x),y)\chi ^{(\alpha
)}{(y)}
\end{equation*}%
with $\chi ^{(\alpha )}{(y)}=1_{\alpha \in (1,2)}+1_{|y|\leqslant
1}1_{\alpha =1}$. Here and throughout the paper we use the standard
convention of summation over repeating indices. The integral part 
\begin{equation*}
\int_{\mathbf{R}_{0}^{d}}\nabla _{y}^{\alpha }u(t,x)m^{(\alpha )}(t,y)\frac{%
dy}{|y|^{d+\alpha }}=\Delta ^{\alpha /2}u
\end{equation*}%
is the fractional Laplacian if $m^{(\alpha )}=1$. The functions $m^{(\alpha
)},l^{(\alpha )}$ are $\mathcal{R}(\mathbb{F})\otimes \mathcal{B}(\mathbf{R}%
_{0}^{d})$-measurable bounded and non-negative; $\sigma ^{i}(t),i=1,\ldots
,d $, are $\mathcal{R}(\mathbb{F})$-measurable\break bounded $Y$-valued
functions, $b(t)=(b^{1}(t),\ldots ,b^{d}(t))$ is a $\mathcal{R}(\mathbb{F})$%
-measurable bounded function and $B(t)=(B^{ij}(t),i,j=1,\ldots ,d)$ is a $%
\mathcal{R}(\mathbb{F})$-measurable bounded symmetric non-negative definite
matrix-valued function. We assume \textbf{parabolicity} of~(\ref{intr1}),
i.e. $m^{(\alpha )}-l^{(\alpha )}\geq 0$ if $\alpha \in (0,2)$ and the
matrix $B^{ij}(t)-\frac{1}{2}\sigma ^{i}(t)\cdot \sigma ^{j}(t)$ is
non-negative definite if $\alpha =2$ ($\cdot $ denotes the inner product in $%
Y$).

The equation (\ref{intr1}) is the model problem for the Zakai equation (see 
\cite{za}) arising in the nonlinear filtering problem. Let $\alpha \in (0,2)$
and $Z_{t}^{i},t\geq 0,i=1,2,$ be two independent $\alpha $-stable processes
defined by 
\begin{equation*}
Z_{t}^{i}=\int_{0}^{t}\int_{\mathbf{R}_{0}^{d}}y\chi ^{(\alpha
)}(y)q^{Z^{i}}(ds,dy)+\int_{0}^{t}\int_{\mathbf{R}_{0}^{d}}y\big[1-\chi
^{(\alpha )}(y)\big]p^{Z^{i}}(ds,dy),
\end{equation*}%
where $p^{Z^{i}}(ds,dy)$ is the jump measure of $Z^{i}$ and 
\begin{equation*}
q^{Z^{i}}(ds,dy)=p^{Z^{i}}(ds,dy)-m_{i}^{(\alpha )}(s,y)\frac{dyds}{%
|y|^{d+\alpha }}
\end{equation*}%
is the martingale measure. Assume that the signal process 
\begin{equation*}
X_{t}=X_{0}+Z_{t}^{1}+Z_{t}^{2},\quad t\geq 0,
\end{equation*}%
and we observe $Y_{t}=Z_{t}^{2}$. Suppose that $X_{0}$ has a probability
density function $u_{0}(x)$ and does not depend on $Z^{i},i=1,2.$ Then for
every function $f\in C_{0}^{\infty }(\mathbf{R}^{d}),$ the optimal mean
square estimate for $f\left( X_{t}\right) ,\,t\in \left[ 0,T\right] $, given
the past of \ the observations $\mathcal{F}_{t}^{Y}=\sigma (Y_{s},s\leqslant
t),$ is of the form $\pi _{t}(f)=\mathbf{E}(f(X_{t})|\mathcal{F}_{t}^{Y})$.
According to \cite{grig1},%
\begin{eqnarray*}
&&d\pi _{t}(f)=\int \pi _{t}\big(f(\cdot +y)-f\big)q^{Y}(dt,dy) \\
&&\quad +\int_{\mathbf{R}_{0}^{d}}\pi _{t}(\nabla _{y}^{\alpha }f(\cdot
))[m_{1}^{(\alpha )}(t,y)+m_{2}^{(\alpha )}(t,y)]\frac{dy}{|y|^{d+\alpha }}%
dt.
\end{eqnarray*}%
Assume there is a smooth $\mathbb{(}\mathcal{F}_{t+}^{Y})$-adapted filtering
density function $v(t,x)$, 
\begin{equation*}
\mathbf{E}\left[ f\left( X_{t}\right) |\mathcal{F}_{t}^{Y}\right] =\int
v\left( t,x\right) f\left( x\right) \,dx,f\in C_{0}^{\infty }(\mathbf{R}%
^{d}).
\end{equation*}%
Integrating by parts, we get%
\begin{eqnarray*}
dv(t,x) &=&\int_{\mathbf{R}_{0}^{d}}[v(t,x+y)-v(t,x)]q^{-Y}(dt,dy) \\
&&+\int_{\mathbf{R}_{0}^{d}}\nabla _{y}^{\alpha }v(t,x)[m_{1}^{(\alpha
)}(t,-y)+m_{2}^{(\alpha )}(t,-y)]\frac{dydt}{|y|^{d+\alpha }}, \\
v(0,x) &=&u_{0}(x).
\end{eqnarray*}

On the other hand, the proof of Proposition \ref{prop2} shows that given $%
\mathcal{F}_{t}^{Y}$ the solution $u(t,x)$ to (\ref{intr1}) with $\Phi =0$
and smooth deterministic input functions $u_{0},f,g,$ is the best mean
square estimate of 
\begin{eqnarray*}
\xi (t,x) &=&u_{0}(x+X_{t}-X_{0})+\int_{0}^{t}f(r,x+X_{t}-X_{r})dr \\
&&+\int_{0}^{t}\int g(r,x+X_{t}-X_{r},y)q^{(\alpha )}(\overleftarrow{d}r,dy)
\end{eqnarray*}%
(here $\overleftarrow{d}$ denotes the backward stochastic integral):%
\begin{equation*}
u(t,x)=\mathbf{E}[\xi (t,x)|\mathcal{F}_{t}^{Y}].
\end{equation*}

The general Cauchy problem for a linear parabolic SPDE of the second order 
\begin{equation}
\left\{ 
\begin{array}{ll}
du=(\frac{1}{2}a^{ij}\,\partial _{ij}u+b^{i}\partial _{i}u+c\,u+f)dt+(\sigma
^{i}\partial _{i}u+h\,u+g)d{W}_{t} & \text{in }E, \\ 
u(0,x)=0 & \text{in }\mathbf{R}^{d}%
\end{array}%
\right.  \label{intr3}
\end{equation}%
driven by a Wiener process $W_{t}$ has been studied by many authors. When
the matrix $(a^{ij}-\sigma ^{i}\cdot \sigma ^{j})$ is uniformly
non-degenerate there exists a complete theory in Sobolev spaces and in the
spaces of Bessel potentials $H_{s}^{p}$ (see \cite{kry1} and references
therein).

In \cite{cl}, the equation (\ref{intr1}) was considered in fractional
Sobolev and Besov spaces in the case of $A^{(\alpha )}=\Delta ^{\alpha /2}$
with $q^{(a)}=0,\eta =0,\sigma =0$ and a finite dimensional $Y$.

In \cite{kim1}, the equation (\ref{intr1}) was considered in fractional
Sobolev spaces in the following special form (see equation (3.4) in \cite%
{kim1}):%
\begin{equation}
du(t)=(a(t)\Delta ^{\alpha /2}u(t)+f(t))dt+\sum_{k=1}^{\infty
}h^{k}(t)dW_{t}^{k}+\sum_{k=1}^{\infty }g^{k}(t)dY^{k},  \label{f01}
\end{equation}%
where $a(t)\geq \delta >0$ is a positive scalar function, $W^{k}$ are
independent standard Wiener processes,%
\begin{equation*}
Y_{t}^{k}=\int_{0}^{t}\int z[N^{k}(ds,dz)-\pi _{k}(dz)ds],t\geq 0,k\geq 1,
\end{equation*}%
are independent $\mathbf{R}^{m}$-valued with independent Poisson point
measures $N^{k}(ds,dz)$ on $[0,\infty )\times \mathbf{R}_{0}^{m},\mathbf{E}%
N^{k}(ds,dz)=\pi _{k}(dz)ds$ and 
\begin{equation*}
\int |z|^{2}\pi _{k}(dz)<\infty ,k\geq 1.
\end{equation*}%
Since $Y_{t}^{k}$ are independent they do not have common jump moments and
we can introduce a point measure $\nu (ds,d\upsilon )$ on $[0,\infty )\times
U$ with $U=\mathbf{N\times R}_{0}^{m}$ ($\mathbb{N}=\{1,2,\ldots \}$) by%
\begin{equation*}
\nu (ds,d\upsilon )=\nu (ds,dkdz)=N^{k}(ds,dz)dk,
\end{equation*}%
where $dk$ is the counting measure on $\mathbf{N}$. Then $\mathbf{E}%
p(ds,dkdz)=$ $\pi _{k}(dz)dkds$ and%
\begin{equation*}
\eta (ds,d\upsilon )=\nu (ds,dkdz)-\pi _{k}(dz)dkds
\end{equation*}%
is a martingale measure. Therefore with $Y=l^{2}$ (the space of square
summable sequences) we can rewrite (\ref{f01}) as%
\begin{equation}
du(t)=(a(t)\Delta ^{\alpha /2}u(t)+f(t))dt+h(t)dW_{t}+\int_{U}\Phi
(t,\upsilon )\eta (dt,d\upsilon )  \label{f02}
\end{equation}%
where $\Phi (t,\upsilon )=g(t,k,z)=g^{k}(t)\cdot z$. Thus (\ref{f01}) is a
partial case of (\ref{intr1}) with $q^{(\alpha )}=0$ and $m^{(\alpha
)}(t,y)=a(t)$. Theorem \ref{main1} below shows that the estimates of the
main Theorem 3.6 in \cite{kim1} are not sharp and the assumptions can be
relaxed. Contrary to the case of a partial differential equation, in order
to handle an equation with $A^{(\alpha )},$ it is not sufficient to consider
an equation with fractional Laplacian like (\ref{f02}). Since only
measurability of $m^{(\alpha )}(t,y)$ in $y$ is assumed, in general (for $%
\alpha \in (0,2)$) the symbol of $A^{(\alpha )}$%
\begin{equation*}
\psi ^{(\alpha )}(t,\xi )=\int \left[ e^{i(\xi ,y)}-1-\chi _{\alpha
}(y)i(\xi ,y)\right] m^{(\alpha )}(t,y)\frac{dy}{|y|^{d+\alpha }}-i(b(t),\xi
)1_{\alpha =1}
\end{equation*}%
is not smooth in $\xi $. In addition, $m^{(\alpha )}(t,y)$ can degenerate on
a substantial set (see Assumption A and Remark \ref{r1} below). The equation
(\ref{intr1}) in H\"{o}lder classes was considered in \cite{MiP09}.

In this paper, we prove the solvability of the general Cauchy model problem (%
\ref{intr1}) in fractional Sobolev spaces. In Section 2, we introduce the
notation and state our main results. In Section 3, we prove some auxiliary
results concerning approximation of the input functions. In Section 4, we
consider a partial case of (\ref{intr1}) with $q^{(\alpha )}=0$, non-random $%
m^{(\alpha )}$ and smooth input functions. In the last two sections we give
the proofs of the main results.

\section{Notation, function spaces and main results}

\subsection{Notation}

The following notation will be used in the paper.

Let $\mathbf{N}_{0}=\{0,1,2,\ldots \},\mathbf{R}_{0}^{d}=\mathbf{R}%
^{d}\backslash \{0\}.$ If $x,y\in \mathbf{R}^{d},$\ we write 
\begin{equation*}
(x,y)=\sum_{i=1}^{d}x_{i}y_{i},\,|x|=\sqrt{(x,x)}.
\end{equation*}

We denote by $C_{0}^{\infty }(\mathbf{R}^{d})$ the set of all infinitely
differentiable functions on $\mathbf{R}^{d}$ with compact support.

We denote the partial derivatives in $x$ of a function $u(t,x)$ on $\mathbf{R%
}^{d+1}$ by $\partial _{i}u=\partial u/\partial x_{i}$, $\partial
_{ij}^{2}u=\partial ^{2}u/\partial x_{i}\partial x_{j}$, etc.;$\,Du=\nabla
u=(\partial _{1}u,\ldots ,\partial _{d}u)$ denotes the gradient of $u$ with
respect to $x$; for a multiindex $\gamma \in \mathbf{N}_{0}^{d}$ we denote%
\begin{equation*}
D_{x}^{{\scriptsize \gamma }}u(t,x)=\frac{\partial ^{|{\scriptsize \gamma |}%
}u(t,x)}{\partial x_{1}^{{\scriptsize \gamma _{1}}}\ldots \partial x_{d}^{%
{\scriptsize \gamma _{d}}}}.
\end{equation*}%
For $\alpha \in (0,2]$ and a function $u(t,x)$ on $\mathbf{R}^{d+1}$, we
write 
\begin{equation*}
\partial ^{{\scriptsize \alpha }}u(t,x)=-\mathcal{F}^{-1}[|\xi |^{%
{\scriptsize \alpha }}\mathcal{F}u(t,\xi )](x),
\end{equation*}%
where 
\begin{equation*}
\mathcal{F}h(t,\xi )=\int_{\mathbf{R}^{d}}\,\mathrm{e}^{-i({\scriptsize \xi
,x)}}h(t,x)dx,\mathcal{F}^{-1}h(t,\xi )=\frac{1}{(2\pi )^{d}}\int_{\mathbf{R}%
^{d}}\,\mathrm{e}^{i({\scriptsize \xi ,x)}}h(t,\xi )d\xi .
\end{equation*}

The letters $C=C(\cdot ,\ldots ,\cdot )$ and $c=c(\cdot ,\ldots ,\cdot )$
denote constants depending only on quantities appearing in parentheses. In a
given context the same letter will (generally) be used to denote different
constants depending on the same set of arguments.

\subsection{Function spaces\label{test}}

Let $\mathcal{S}(\mathbf{R}^{d})$ be the Schwartz space of smooth
real-valued rapidly decreasing functions. Let $V$ be a Banach space with a
norm $|\cdot |_{V}$. The space of $V$-valued tempered distributions we
denote by $\mathcal{S}^{\prime }(\mathbf{R}^{d},V)$ ($f\in \mathcal{S}%
^{\prime }(\mathbf{R}^{d},V)$ is a continuous $V$-valued linear functional
on $\mathcal{S}(\mathbf{R}^{d})$). If $V=\mathbf{R}$, we write $\mathcal{S}%
^{\prime }(\mathbf{R}^{d},V)=\mathcal{S}^{\prime }(\mathbf{R}^{d})$ and
denote by $\left\langle \cdot ,\cdot \right\rangle $ the duality between $%
\mathcal{S}^{\prime }(\mathbf{R}^{d})$ and $\mathcal{S}(\mathbf{R}^{d}).$

For a $V$-valued measurable function $h$ on $\mathbf{R}^{d}$ and $%
p\geqslant1 $ we denote 
\begin{equation*}
|h|_{V,p}^{p}=\int_{\mathbf{R}^{d}}|h(x)|_{V}^{p}dx.
\end{equation*}

Further, for a characterization of our function spaces we will use the
following construction (see \cite{BeLo76}). By Lemma 6.1.7 in \cite{BeLo76},
there is a function $\phi \in C_{0}^{\infty }(\mathbf{R}^{d})$ such that $%
\mathrm{supp}\,\phi =\{\xi :\frac{1}{2}\leqslant |\xi |\leqslant 2\}$, $\phi
(\xi )>0$ if $2^{-1}<|\xi |<2$ and 
\begin{equation*}
\sum_{j=-\infty }^{\infty }\phi (2^{-j}\xi )=1\quad \text{if }\xi \neq 0.
\end{equation*}%
Define the functions $\varphi _{k}\in \mathcal{S}(\mathbf{R}^{d}),$ $%
k=1,\ldots ,$ by 
\begin{equation*}
\mathcal{F}\varphi _{k}(\xi )=\phi (2^{-k}\xi ),
\end{equation*}%
and $\varphi _{0}\in \mathcal{S}(\mathbf{R}^{d})$ by 
\begin{equation*}
\mathcal{F}\varphi _{0}(\xi )=1-\sum_{k\geqslant 1}\mathcal{F}\varphi
_{k}(\xi ).
\end{equation*}

Let $\beta \in \mathbf{R}$ and $p\geqslant 1$. We introduce the Besov space $%
B_{pp}^{{\scriptsize \beta }}(\mathbf{R}^{d},V)$ of generalized functions $%
f\in \mathcal{S}^{\prime }(\mathbf{R}^{d},V)$ with finite norm%
\begin{equation*}
|f|_{B_{pp}^{{\scriptsize \beta }}(\mathbf{R}^{d},V)}=\Bigg\{%
\sum_{j=0}^{\infty }2^{j{\scriptsize \beta }p}|\varphi _{j}\ast f|_{V,p}^{p}%
\Bigg\}^{1/p}
\end{equation*}%
and the Sobolev space $H_{p}^{{\scriptsize \beta }}(\mathbf{R}^{d},V)$ of $%
f\in \mathcal{S}^{\prime }(\mathbf{R}^{d},V)$ with finite norm%
\begin{equation}
|f|_{H_{p}^{{\scriptsize \beta }}(\mathbf{R}^{d},V)}=|\mathcal{F}%
^{-1}((1+|\xi |^{2})^{{\scriptsize \beta /2}}\mathcal{F}f)|_{V,p}=\left\vert
J^{\beta }f\right\vert _{V,p},  \label{nf5}
\end{equation}%
where 
\begin{equation*}
J^{\beta }=(I-\Delta )^{\beta /2},
\end{equation*}%
$I$ is the identity map and $\Delta $ is the Laplacian in $\mathbf{R}^{d}$.
For the scalar functions an equivalent norm to (\ref{nf5}) is defined by%
\begin{equation}
|f|_{H_{p}^{\beta }(\mathbf{R}^{d})}=\Bigg\{\int_{\mathbf{R}^{d}}\bigg(%
\sum_{j=0}^{\infty }2^{2\beta j}|\varphi _{j}\ast f(x)|^{2}\bigg)^{p/2}dx%
\Bigg\}^{1/p}.  \label{nf0}
\end{equation}

We also introduce the corresponding spaces of generalized functions on $%
E=[0,T]\times\mathbf{R}^d$. The spaces $B^{\beta}_{pp}(E,V)$ and $%
H^{\beta}_{p}(E,V)$ consist of all measurable $S^{\prime }(\mathbf{R}^d,V)$%
-valued functions $f$ on $[0,T]$ with finite norms 
\begin{equation*}
|f|_{B_{pp}^{{\scriptsize \beta }}(E,V)}=\left\{\int_0^T|f(t)|^p_{B^{%
\beta}_{pp}(\mathbf{R}^d,V)}dt\right\}^{\frac1p}
\end{equation*}
and 
\begin{equation*}
|f|_{H_{p}^{{\scriptsize \beta }}(E,V)}=\left\{\int_0^T|f(t)|^p_{H^{%
\beta}_{p}(\mathbf{R}^d,V)} dt\right\}^{\frac1p}.
\end{equation*}

Similarly we introduce the corresponding spaces of random generalized
functions.

Let $(\Omega, \mathcal{F},\mathbf{P})$ be a complete probability space with
a filtration of $\sigma$-algebras $\mathbb{F}=(\mathcal{F}_t)$ satisfying
the usual conditions. Let $\mathcal{R}(\mathbb{F})$ be the progressive $%
\sigma$-algebra on $[0,\infty)\times\Omega$.

The spaces $\mathbb{B}_{pp}^{{\scriptsize \beta }}(\mathbf{R}^d,V)$ and $%
\mathbb{H}_{p}^{{\scriptsize \beta }}(\mathbf{R}^d,V)$ consist of all $%
\mathcal{F}$-measurable random functions $f$ with values in ${B}_{pp}^{%
{\scriptsize \beta }}(\mathbf{R}^d,V)$ and ${H}_{p}^{{\scriptsize \beta }}(%
\mathbf{R}^d,V)$ with finite norms 
\begin{equation*}
|f|_{\mathbb{B}_{pp}^{{\scriptsize \beta }}(\mathbf{R}^d,V)}= \Big\{\mathbf{E%
}|f|^p_{B_{pp}^{{\scriptsize \beta }}(\mathbf{R}^{d},V)}\Big\}^{1/p}
\end{equation*}
and 
\begin{equation*}
|f|_{\mathbb{H}_{p}^{{\scriptsize \beta }}(\mathbf{R}^d,V)}= \Big\{\mathbf{E}%
|f|^p_{H_{p}^{{\scriptsize \beta }}(\mathbf{R}^{d},V)}\Big\}^{1/p}.
\end{equation*}

The spaces $\mathbb{B}_{pp}^{{\scriptsize \beta }}(E,V)$ and $\mathbb{H}%
_{p}^{{\scriptsize \beta }}(E,V)$ consist of all $\mathcal{R}(\mathbb{F})$%
-measurable random functions with values in ${B}_{pp}^{{\scriptsize \beta }%
}(E,V)$ and ${H}_{p}^{{\scriptsize \beta }}(E,V)$ with finite norms%
\begin{equation*}
|f|_{\mathbb{B}_{pp}^{{\scriptsize \beta }}(E,V)}= \Big\{\mathbf{E}%
|f|^p_{B_{pp}^{{\scriptsize \beta }}(E,V)}\Big\}^{1/p}
\end{equation*}
and 
\begin{equation*}
|f|_{\mathbb{H}_{p}^{{\scriptsize \beta }}(E,V)}= \Big\{\mathbf{E}%
|f|^p_{H_{p}^{{\scriptsize \beta }}(E,V)}\Big\}^{1/p}.
\end{equation*}

If $V=L_{r}(U,\mathcal{U},\Pi )$, $r\geq 1$, the space of $r$-integrable
measurable functions on $U$, for brevity of notation we write 
\begin{eqnarray*}
&&B_{r,pp}^{\beta }(A)=B_{pp}^{\beta }(A,V),\quad \mathbb{B}_{r,pp}^{\beta
}(A)=\mathbb{B}_{pp}^{\beta }(A,V), \\
&&H_{r,p}^{\beta }(A)=H_{p}^{\beta }(A,V),\quad \mathbb{H}_{r,p}^{\beta }(A)=%
\mathbb{H}_{p}^{\beta }(A,V), \\
&&L_{r,p}(A)=H_{r,p}^{0}(A),\quad \mathbb{L}_{r,p}(A)=\mathbb{H}%
_{r,p}^{0}(A),
\end{eqnarray*}%
where $A=\mathbf{R}^{d}$ or $E$. For scalar functions we drop $V$ in the
notation of function spaces.

We also introduce the spaces $\overline{\mathbb{B}}_{r,pp}^{\beta }(E)$ and $%
\overline{\mathbb{H}}_{r,p}^{\beta }(E)$ consisting of $\mathcal{R}(\mathbb{F%
})\otimes \mathcal{B}(\mathbf{R}_{0}^{d})$-measurable $S^{\prime }(\mathbf{R}%
^{d})$-valued random functions $f=f(t,x,y)$ with finite norms 
\begin{equation*}
|f|_{\overline{\mathbb{B}}_{r,pp}^{\beta }(E)}=\Bigg\{\mathbf{E}%
\sum_{j=0}^{\infty }2^{j\beta p}\int_{0}^{T}\int_{\mathbf{R}^{d}}||(\varphi
_{j}\ast f)(t,x,\cdot )||_{r}^{p}dxdt\Bigg\}^{1/p}
\end{equation*}%
and 
\begin{equation*}
|f|_{\overline{\mathbb{H}}_{r,p}^{\beta }(E)}=\Bigg\{\mathbf{E}%
\int_{0}^{T}\int_{\mathbf{R}^{d}}||J^{\beta }f(t,x,\cdot )||_{r}^{p}dxdt%
\Bigg\}^{1/p},
\end{equation*}%
where 
\begin{equation*}
||g(t,x,\cdot )||_{r}=\Bigg\{\int_{\mathbf{R}_{0}^{d}}|g(t,x,y)|^{r}l^{(%
\alpha )}(t,y)\frac{dy}{|y|^{d+\alpha }}\Bigg\}^{1/r}.
\end{equation*}

\subsection{Main results}

We fix non-random functions $m_{0}^{(\alpha )}(t,y)\geq 0$, $\alpha \in
(0,2) $, on $[0,T]\times \mathbf{R}_{0}^{d}$ and positive constants $K$ and $%
\delta $. Throughout the paper we assume that the functions $m_{0}^{(\alpha
)}$ satisfy the following conditions. \medskip

\noindent \textbf{Assumption} $\mathbf{A_{0}}.$ (i) For each $\alpha \in
(0,2)$ the function $m_{0}^{(\alpha )}(t,y)\geq 0$ is measurable,
homogeneous in $y$ with index zero, differentiable in $y$ up to the order $%
d_{0}=[\frac{d}{2}]+2$ and%
\begin{equation*}
|D_{y}^{\gamma }m_{0}^{(\alpha )}(t,y)|\leq K
\end{equation*}%
for all $t\in \lbrack 0,T]$, $y\in \mathbf{R}_{0}^{d}$ and multiindices $%
\gamma \in \mathbf{N}_{0}^{d}$ such that $|\gamma |\leq d_{0}$;

(ii) For all $t\in[0,T]$ 
\begin{equation*}
\int_{S^{d-1}}wm_{0}^{(1)}(t,w)\mu _{d-1}(dw)=0,
\end{equation*}
where $S^{d-1}$ is the unit sphere in $\mathbf{R}^{d}$ and $\mu _{d-1}$ is
the Lebesgue measure on it;

(iii) For each $\alpha\in(0,2)$ and $t\in[0,T]$ 
\begin{equation*}
\inf_{|\xi|=1}\int_{S^{d-1}}|(w,\xi)|^{\alpha}m_{0}^{(\alpha)}(t,w)\mu
_{d-1}(dw)\geq\delta>0.
\end{equation*}

\begin{remark}
\label{r1}\emph{The nondegenerateness assumption $A_{0}$ (iii) holds with
certain $\delta >0$ if, e.g.%
\begin{equation*}
\inf_{t\in \lbrack 0,T],w\in \Gamma }m_{0}^{(\alpha )}(t,w)>0
\end{equation*}%
for a measurable subset $\Gamma \subset S^{d-1}$ of positive Lebesgue
measure (the function }$m_{0}^{(\alpha )}$ \emph{can degenerate on a
substantial set).}
\end{remark}

\noindent\textbf{Assumption} $\mathbf{A.}$ (i) The real-valued random
functions $m^{(\alpha)}(t,y)$ and $l^{(\alpha)}(t,y)$ on $[0,T]\times\mathbf{%
R}_{0}^{d}$ are non-negative and $\mathcal{R}(\mathbb{F})\otimes\mathcal{B}(%
\mathbf{R}^d_0)$-measurable; the real-valued random functions $%
B^{ij}(t)=B^{ji}(t),\ b^i(t),\ i,j=1,\ldots,d,$ on $[0,T]$ are $\mathcal{R}(%
\mathbb{F})$-measurable; the Hilbert space $Y$-valued random functions $%
\sigma^i(t),\ i=1,\ldots,d,$ on $[0,T]$ are $\mathcal{R}(\mathbb{F})$%
-measurable.

\smallskip (ii) $\mathbf{P}$-a.s. for all $t\in \lbrack 0,T],\ y\in \mathbf{R%
}_{0}^{d}$ and $i,j=1,\ldots ,d$ 
\begin{equation*}
m^{(\alpha )}(t,y)+l^{(\alpha )}(t,y)+|B^{ij}(t)|+|b^{i}(t)|+|\sigma
^{i}(t)|_{Y}\leq K
\end{equation*}%
and for all $0<r<R<\infty ,$%
\begin{equation*}
\int_{r\leq |y|\leq R}ym^{(1)}(t,y)\frac{dy}{|y|^{d+1}}=0.
\end{equation*}

(iii) $\mathbf{P}$-a.s. for all $t\in \lbrack 0,T]$ and $y\in \mathbf{R}%
_{0}^{d}$ 
\begin{eqnarray*}
&&m^{(\alpha )}(t,y)-l^{(\alpha )}(t,y)\geq m_{0}^{(\alpha )}(t,y)%
\mbox{\quad if }\alpha \in (0,2), \\
&&\big(B^{ij}(t)-\frac{1}{2}\sigma ^{i}(t)\cdot \sigma ^{j}(t)\big)%
y_{i}y_{j}\geq \delta |y|^{2}\mbox{\quad if }\alpha =2,
\end{eqnarray*}%
where $\delta >0$, the function $m_{0}^{(\alpha )}$ satisfies Assumption A$%
_{0}$ and $\cdot $ denotes the inner product in $Y$.

%%%%%%%%%%%%%%%%%

\begin{remark}
\label{ren4}\emph{Assumption A (iii) is called \textbf{superparabolicity} of
(\ref{intr1}).}
\end{remark}

\begin{definition}
\label{def1}Let $\alpha \in (0,2],\ \beta \in \mathbf{R},\ p\geq 2,\
u_{0}\in \mathbb{B}_{pp}^{\beta +\alpha -\frac{\alpha }{p}}(\mathbf{R}^{d})$
be $\mathcal{F}_{0}$-measurable, $f\in \mathbb{H}_{p}^{\beta }(E),\ \Phi \in 
\mathbb{B}_{p,pp}^{\beta +\alpha -\frac{\alpha }{p}}(E)\cap \mathbb{H}%
_{2,p}^{\beta +\frac{\alpha }{2}}(E),\ g\in \mathbb{\bar{B}}_{p,pp}^{\beta
+\alpha -\frac{\alpha }{p}}(E)\cap \mathbb{\bar{H}}_{2,p}^{\beta +\frac{%
\alpha }{2}}(E)$ and $h\in \mathbb{H}_{p}^{\beta +\alpha /2}(E,Y)$.

We say that $u\in \mathbb{H}_{p}^{\beta +\alpha }(E)$ is a strong solution
to \emph{(\ref{intr1})} if $u(t,\cdot )$ is strongly cadlag in $H_{p}^{\beta
}(\mathbf{R}^{d})$ with respect to $t$, $A^{(\alpha )}u\in \mathbb{H}%
_{p}^{\beta }(E)$ and $\mathbf{P}$-a.s. in $\varphi \in \mathcal{S}(\mathbf{R%
}^{d})$ 
\begin{eqnarray}
d\left\langle {u}(t),\varphi \right\rangle &=&\left\langle A^{(\alpha )}{u(t)%
}-\lambda {u(t)}+{f,\varphi }\right\rangle dt  \label{for3} \\
&&+\int_{\mathbf{R}_{0}^{d}}\left\langle {u}(t-,\cdot +y)-{u}(t-,\cdot
)+g(t,\cdot ,y),\varphi \right\rangle q^{(\alpha )}(dt,dy)1_{\alpha \in
(0,2)}  \notag \\
&&+\int_{U}\left\langle {\Phi }(t,\cdot ,\upsilon ),\varphi \right\rangle
\eta (dt,d\upsilon )+\left\langle 1_{\alpha =2}\sigma ^{i}(t)\partial _{i}{u}%
(t)+{h}(t),\varphi \right\rangle dW_{t}\,,  \notag \\
{u}(0) &=&{u}_{0};  \notag
\end{eqnarray}%
equivalently, in integral form%
\begin{eqnarray*}
\left\langle {u}(t),\varphi \right\rangle &=&\left\langle u_{0},\varphi
\right\rangle +\int_{0}^{t}\left\langle A^{(\alpha )}{u(s)}-\lambda {u(s)}+{%
f,\varphi }\right\rangle ds \\
&&+\int_{0}^{t}\int_{\mathbf{R}_{0}^{d}}\left\langle {u}(s-,\cdot +y)-{u}%
(s-,\cdot )+g(s,\cdot ,y),\varphi \right\rangle q^{(\alpha
)}(ds,dy)1_{\alpha \in (0,2)} \\
&&+\int_{0}^{t}\int_{U}\left\langle {\Phi }(s,\cdot ,\upsilon ),\varphi
\right\rangle \eta (ds,d\upsilon )+\int_{0}^{t}\left\langle 1_{\alpha
=2}\sigma ^{i}(s)\partial _{i}{u}(s)+{h}(s),\varphi \right\rangle dW_{s}\,,
\end{eqnarray*}%
${0\leq t\leq T.}$
\end{definition}

\begin{remark}
\label{re}Since according to Theorem 2.4.2 in \cite{tri} 
\begin{equation*}
\mathbb{B}_{p,pp}^{\beta +\alpha -\frac{\alpha }{p}}(E)\cap \mathbb{H}%
_{2,p}^{\beta +\frac{\alpha }{2}}(E)\subseteq \mathbb{H}_{p,p}^{\beta
}(E)\cap \mathbb{H}_{2,p}^{\beta }(E),\mathbb{\bar{B}}_{p,pp}^{\beta +\alpha
-\frac{\alpha }{p}}(E)\cap \mathbb{\bar{H}}_{2,p}^{\beta +\frac{\alpha }{2}%
}(E)\subseteq \mathbb{\bar{H}}_{p,p}^{\beta }(E)\cap \mathbb{\bar{H}}%
_{2,p}^{\beta }(E),
\end{equation*}%
the assumptions of Definition \ref{def1} imply (see Lemma \ref{ler1} below)
that 
\begin{eqnarray*}
M_{t}^{1} &=&\int_{0}^{t}\int_{\mathbf{R}_{0}^{d}}g(s,\cdot ,y)q^{(\alpha
)}(ds,dy)1_{\alpha \in (0,2)}, \\
M_{t}^{2} &=&\int_{0}^{t}\int_{U}{\Phi }(s,\cdot ,\upsilon )\eta
(ds,d\upsilon ),0\leq t\leq T,
\end{eqnarray*}%
are cadlag $H_{p}^{\beta }(\mathbf{R}^{d})$-valued. According to Theorem 1
in \cite{mir3mol}, 
\begin{equation*}
M_{t}^{3}=\int_{0}^{t}[1_{\alpha =2}\sigma ^{i}(s)\partial _{i}J^{\beta }{u}%
(s)+J^{\beta }{h}(s)]dW_{s},0\leq t\leq T,
\end{equation*}%
is continuous $H_{p}^{\beta }(\mathbf{R}^{d})$-valued. By Corollary \ref%
{corr1} below 
\begin{equation*}
Q(t)=\int_{0}^{t}\int \left[ u(s,\cdot +y)-u(s,\cdot )\right] q^{(\alpha
)}(ds,dy),0\leq t\leq T,
\end{equation*}%
is $H_{p}^{\beta }(\mathbf{R}^{d})$-valued cadlag.
\end{remark}

The first main result concerns the so-called uncorrelated case of (\ref%
{intr1}) defined by \ 
\begin{eqnarray}
du(t,x) &=&\bigl(A^{(\alpha )}u-\lambda u+f\bigr)(t,x)dt  \label{equ2} \\
&&+\int_{U}\Phi (t,x,\upsilon )\eta (dt,d\upsilon )+h(t,x)dW_{t}\text{\quad
in }E,  \notag \\
u(0,x) &=&u_{0}(x)\text{\quad in }\mathbf{R}^{d}.  \notag
\end{eqnarray}%
%
%
%
%
%
%
%
%
%
%
%
%
%
%
%
%
%
%
%
%
%
%
%
%
%
%\begin{equation*}
%\sup_{t,|\xi |=1}\func{Re}\tilde{\psi}_{0}^{(\alpha )}(t,\xi )\leq -\mu ,
%\end{equation*}%
%where
%\begin{eqnarray*}
%{}\tilde{\psi}_{0}^{(\alpha )}(t,\xi ) &=&-\int_{S^{d-1}}|(w,\xi )|^{\alpha }%
%\Big[1-i\Big(\tan \frac{\alpha \pi }{2}\mbox{sgn}(w,\xi )1_{\alpha \neq 1} \\
%&&-\frac{2}{\pi }\mbox{sgn}(w,\xi )\ln |(w,\xi )|1_{\alpha =1}\Big)\Big]%
%m_{0}^{(\alpha )}(t,w)\mu _{d-1}(dw),
%\end{eqnarray*}%
The following statement is a consequence of Theorem \ref{thm15} proved in
Section \ref{sei} below.

\begin{theorem}
\label{main1}Let $\alpha \in (0,2],\ \beta \in \mathbf{R},\ p\geq 2$ and
Assumption~\emph{A} be satisfied with $l^{(\alpha )}=0$ and $\sigma ^{i}=0,\
i=1,\ldots ,d$. Let $u_{0}\in \mathbb{B}_{pp}^{\beta +\alpha -\frac{\alpha }{%
p}}(\mathbf{R}^{d})$ be $\mathcal{F}_{0}$-measurable, $f\in \mathbb{H}%
_{p}^{\beta }(E),\ \Phi \in \mathbb{B}_{p,pp}^{\beta +\alpha -\frac{\alpha }{%
p}}(E)\cap \mathbb{H}_{2,p}^{\beta +\frac{\alpha }{2}}(E)$ and $h\in \mathbb{%
H}_{p}^{\beta +\alpha /2}(E,Y)$.

Then there is a unique strong solution $u\in \mathbb{H}_{p}^{\beta +\alpha
}(E)$ of \emph{(\ref{equ2})}. Moreover, there is a constant $C=C(\alpha
,\beta ,p,d,T,K,\delta )$ such that 
\begin{eqnarray}
|u|_{\mathbb{H}_{p}^{\beta +\alpha }(E)} &\leq &C\Bigl(|u_{0}|_{\mathbb{B}%
_{pp}^{\beta +\alpha -\frac{\alpha }{p}}(E)}+|f|_{\mathbb{H}_{p}^{\beta }(E)}
\label{mf1} \\
&&+|\Phi |_{\mathbb{H}_{2,p}^{\beta +\frac{\alpha }{2}}(E)}+|\Phi |_{\mathbb{%
B}_{p,pp}^{\beta +\alpha -\frac{\alpha }{p}}(E)}+|h|_{\mathbb{H}_{p}^{\beta
+\alpha /2}(E,Y)}\Bigr).  \notag
\end{eqnarray}
\end{theorem}

\begin{remark}
\emph{According to the embedding theorem (see Theorem 6.4.4 in \cite{BeLo76}%
), the estimate \ (\ref{mf1}) holds with $|u_{0}|_{\mathbb{B}_{pp}^{\kappa
}(E)}$ and $|\Phi |_{\mathbb{B}_{p,pp}^{\kappa }(E)}$ replaced by $|u_{0}|_{%
\mathbb{H}_{p}^{\kappa }(E)}$ and $|\Phi |_{\mathbb{H}_{p,p}^{\kappa }(E)}$,
where $\kappa ={\beta +\alpha -\frac{\alpha }{p}}$.}
\end{remark}

Theorem \ref{main1} covers the deterministic equation%
\begin{eqnarray}
\partial _{t}u(t,x) &=&A^{(\alpha )}u(t,x)-\lambda u(t,x)+f(t,x)\text{\quad
in }E,  \label{equ3} \\
u(0,x) &=&u_{0}(x)\text{\quad in }\mathbf{R}^{d}  \notag
\end{eqnarray}%
with non-random coefficients and input functions. The following obvious
consequence of Theorem \ref{main1} holds.

\begin{corollary}
\label{corn2}Let $\alpha\in(0,2],\ \beta \in \mathbf{R},\ p\geq 2$ and for
all $t\in [0,T],\ y\in \mathbf{R}_{0}^{d},$%
\begin{eqnarray*}
&&m^{(\alpha )}(t,y)+|B^{ij}(t)|+|b^{i}(t)|\leq K,\quad i,j=1,\ldots,d, \\
&&m^{(\alpha )}(t,y) \geq m_{0}^{(\alpha )}(t,y)\text{\quad if }\alpha \in
(0,2)
\end{eqnarray*}%
and 
\begin{equation*}
\hspace*{-3.3cm}B^{ij}(t)y _{i}y _{j} \geq \delta |y |^{2}\text{\quad if }%
\alpha =2,
\end{equation*}%
where $m_{0}^{(\alpha )}$ satisfies Assumption \emph{A}$_0$. Let $u_{0}\in {B%
}_{pp}^{\beta +\alpha -\frac{\alpha }{p}}(\mathbf{R}^d)$ and $f\in {H}%
_{p}^{\beta }(E)$.

Then there is a unique strong solution $u\in {H}_{p}^{\beta +\alpha }(E)$ of 
\emph{(\ref{equ3})}. Moreover, there is a constant $C=C(\alpha ,\beta
,p,d,T,K,\delta )$ such that 
\begin{equation*}
|u|_{{H}_{p}^{\beta +\alpha }(E)}\leq C\left( |u_{0}|_{{B}_{pp}^{\beta
+\alpha -\frac{\alpha }{p}}(E)}+|f|_{{H}_{p}^{\beta }(E)}\right) .
\end{equation*}
\end{corollary}

Given $g\in \mathbb{\bar{B}}_{p,pp}^{\beta +\alpha -\frac{\alpha }{p}%
}(E)\cap \mathbb{\bar{H}}_{2,p}^{\beta +\frac{\alpha }{2}}(E)$ we denote 
\begin{equation*}
\Lambda g(t,x,y)=g(t,x-y,y),(t,x)\in E,y\in \mathbf{R}_{0}^{d}.
\end{equation*}%
For a $g,\Lambda g\in \mathbb{\bar{B}}_{p,pp}^{\beta +\alpha -\frac{\alpha }{%
p}}(E)\cap \mathbb{\bar{H}}_{2,p}^{\beta +\frac{\alpha }{2}}(E)$ we denote%
\begin{equation*}
Ig(t,x)=\int_{\mathbf{R}_{0}^{d}}(\Lambda g-g)(t,x,y)l^{(\alpha )}(t,y)\frac{%
dy}{|y|^{d+\alpha }}(t,x)\in E,
\end{equation*}%
assuming that 
\begin{equation}
Ig(t,x)=\lim_{\varepsilon \rightarrow 0}\int_{|y|>\varepsilon }(\Lambda
g-g)(t,x,y)l^{(\alpha )}(t,y)\frac{dy}{|y|^{d+\alpha }}(t,x)\in E,
\label{form1}
\end{equation}%
is well defined as a limit in $\mathbb{H}_{p}^{\beta }(E)$ (we write simply
that $Ig\in \mathbb{H}_{p}^{\beta }(E)$ in this case).

In the general case the following statement holds for (\ref{intr1}).

\begin{theorem}
\label{main 2}Let $\alpha \in (0,2],\ \beta \in \mathbf{R},\ p\geq 2$ and
Assumption~\emph{A} be satisfied. Let $u_{0}\in \mathbb{B}_{pp}^{\beta
+\alpha -\frac{\alpha }{p}}(\mathbf{R}^{d})$ be $\mathcal{F}_{0}$%
-measurable, $f,Ig\in \mathbb{H}_{p}^{\beta }(E),\Phi \in \mathbb{B}%
_{p,pp}^{\beta +\alpha -\frac{\alpha }{p}}(E)\cap \mathbb{H}_{2,p}^{\beta +%
\frac{\alpha }{2}}(E)$, $g,\Lambda g\in \mathbb{\bar{B}}_{p,pp}^{\beta
+\alpha -\frac{\alpha }{p}}(E)\cap \mathbb{\bar{H}}_{2,p}^{\beta +\frac{%
\alpha }{2}}(E)$ and $h\in \mathbb{H}_{p}^{\beta +\alpha /2}(E,Y)$.

Then there is a unique strong solution $u\in \mathbb{H}_{p}^{\beta +\alpha
}(E)$ of \emph{(\ref{intr1})}. Moreover, there is a constant $C=C(\alpha
,\beta ,p,d,T,K,\delta )$ such that 
\begin{eqnarray*}
|u|_{\mathbb{H}_{p}^{\beta +\alpha }(E)} &\leq &C\Bigl(|u_{0}|_{\mathbb{B}%
_{pp}^{\beta +\alpha -\frac{\alpha }{p}}(E)}+|f+Ig|_{\mathbb{H}_{p}^{\beta
}(E)} \\
&&+|h|_{\mathbb{H}_{p}^{\beta +\alpha /2}(E,Y)}+|\Phi |_{\mathbb{H}%
_{2,p}^{\beta +\frac{\alpha }{2}}(E)}+|\Phi |_{\mathbb{B}_{p,pp}^{\beta
+\alpha -\frac{\alpha }{p}}(E)} \\
&&+|\Lambda g|_{\mathbb{\bar{H}}_{2,p}^{\beta +\frac{\alpha }{2}%
}(E)}+|\Lambda g|_{\mathbb{\bar{B}}_{p,pp}^{\beta +\alpha -\frac{\alpha }{p}%
}(E)}\Bigr).
\end{eqnarray*}
\end{theorem}

\section{Approximation of input functions}

Let a non-negative function $\zeta\in C_{0}^{\infty }(\mathbf{R}^{d})$ be
such that $\zeta(x)=0$ if $|x|\geq 1$ and $\int \zeta(x) dx=1$. For $%
\varepsilon \in (0,1)$ we set 
\begin{equation*}
\zeta_{\varepsilon }(x)=\varepsilon ^{-d}\zeta(x/\varepsilon ),\quad x\in 
\mathbf{R}^{d}.
\end{equation*}%
Let $V$ be a Banach space with a norm $|\cdot|_V$.

\begin{lemma}
\label{cnew}Let $\beta\in\mathbf{R},\ p\geq1$ and $u\in A$, where $%
A=B_{pp}^{\beta }(\mathbf{R}^{d},V)$ or $H_{p}^{\beta }(\mathbf{R}^{d},V)$.
Let $u_{\varepsilon }=u\ast \zeta_{\varepsilon }$.

Then $|u_{\varepsilon }-u|_{A} \rightarrow 0 \mbox{ as }\varepsilon \to 0$.
Moreover, for every $\varepsilon $ and multiindex $\gamma \in \mathbf{N}%
_{0}^{d}$ there is a constant $C$ not depending on $u$ such that%
\begin{equation*}
\sup_{x}|\partial ^{\gamma }u_{\varepsilon }(x)|_V + |\partial ^{\gamma
}u_{\varepsilon }|_{V,p} \leq C|u|_{A}.
\end{equation*}
\end{lemma}

\begin{proof}
Let $v\in L_p(\mathbf{R}^d,V)$, $p\geq1$, and $v_{\varepsilon}=v\ast
\zeta_{\varepsilon}$. It is well known that 
\begin{equation*}
|v_{\varepsilon}|_{V,p}\leq |v|_{V,p}
\end{equation*}
and 
\begin{equation*}
|v-v_{\varepsilon}|_{V,p}\to0,\quad \varepsilon\to0.
\end{equation*}
Therefore, 
\begin{eqnarray*}
|u_{\varepsilon }|_{B_{pp}^{\beta }(\mathbf{R}^{d},V)}^{p}
&=&\sum_{j=0}^{\infty }2^{j\beta p}|\varphi _{j}\ast u_{\varepsilon
}|_{V,p}^{p} =\sum_{j=0}^{\infty }2^{j\beta p}|(\varphi _{j}\ast
u)_{\varepsilon }|_{V,p}^{p} \\
&\leq &\sum_{j=0}^{\infty }2^{j\beta p}|\varphi _{j}\ast
u|_{V,p}^{p}=|u|^p_{B_{pp}^{\beta }(\mathbf{R}^{d},V)}
\end{eqnarray*}%
and 
\begin{eqnarray*}
|u_{\varepsilon }-u|_{B_{pp}^{\beta }(\mathbf{R}^{d},V)}^{p}&=&\sum_{j=0}^{%
\infty }2^{j\beta p}|\varphi _{j}\ast (u_{\varepsilon }-u)|_{V,p}^{p} \\
&=&\sum_{j=0}^{\infty }2^{j\beta p}\big\vert(\varphi _{j}\ast
u)_{\varepsilon }-\varphi_{j}\ast u\big\vert_{V,p}^{p}\to 0
\end{eqnarray*}%
as $\varepsilon\to 0$. Similarly,%
\begin{equation*}
|u_{\varepsilon }|_{H_{p}^{\beta }(\mathbf{R}^{d},V)}=|J^{\beta
}u_{\varepsilon }|_{V,p}=|\left(J^{\beta}u\right) _{\varepsilon }|_{V,p}\leq
|J^{\beta}u|_{V,p}=|u|_{H_p^{\beta}(\mathbf{R}^d,V)}
\end{equation*}%
and%
\begin{equation*}
|u_{\varepsilon }-u|_{H_{p}^{\beta }(\mathbf{R}^{d},V)}
=|J^{\beta}(u_{\varepsilon }-u)|_{V,p}=\left\vert \left( J^{\beta}u\right)
_{\varepsilon }-J^{\beta}u\right\vert _{V,p}\to0
\end{equation*}%
as $\varepsilon \rightarrow 0$.

Let $u\in B_{pp}^{\beta }(\mathbf{R}^{d},V)$. Then 
\begin{equation*}
u=\sum_{j=0}^{\infty }u\ast \varphi _{j}
\end{equation*}%
in $\mathcal{S}^{\prime }(\mathbf{R}^{d})$. Therefore, 
\begin{equation*}
u_{\varepsilon }=\sum_{j=0}^{\infty }u\ast \varphi _{j}\ast \zeta
_{\varepsilon }
\end{equation*}%
and, for every $m\in \mathbf{N}_{0}$ and $l>0$, 
\begin{equation*}
J^{m}u_{\varepsilon }=\sum_{j=0}^{\infty }J^{-l}u\ast \varphi _{j}\ast
J^{m+l}\zeta _{\varepsilon }.
\end{equation*}%
Applying Minkowski's and H\"{o}lder's inequalities, we get 
\begin{eqnarray*}
|J^{m}u_{\varepsilon }(x)|_{V} &\leq &\sum_{j=0}^{\infty }\big\vert %
J^{-l}u\ast \varphi _{j}\ast J^{m+l}\zeta _{\varepsilon }\big\vert_{V}\leq
\sum_{j=0}^{\infty }\big\vert J^{-l}u\ast \varphi _{j}\big\vert_{V}\ast %
\big\vert J^{m+l}\zeta _{\varepsilon }\big\vert(x) \\
&\leq &C\sum_{j=0}^{\infty }\big\vert J^{-l}u\ast \varphi _{j}\big\vert_{V,p}
\end{eqnarray*}%
and 
\begin{equation*}
|J^{m}u_{\varepsilon }|_{V,p}\leq \sum_{j=0}^{\infty }\big\vert J^{-l}u\ast
\varphi _{j}\ast J^{m+l}\zeta _{\varepsilon }\big\vert_{V,p}\leq
C\sum_{j=0}^{\infty }\big\vert J^{-l}u\ast \varphi _{j}\big\vert_{V,p}.
\end{equation*}%
By Lemma 6.2.1 in \cite{BeLo76}, 
\begin{equation*}
|J^{-l}u\ast \varphi _{j}|_{V,p}\leq C2^{-lj}|u\ast \varphi _{j}|_{V,p}.
\end{equation*}

On the other hand, 
\begin{equation*}
|u\ast\varphi _{j}|_{V,p}\leq 2^{-\beta j}|u|_{B_{pp}^{\beta }(\mathbf{R}%
^{d},V)},\quad j\geq 0.
\end{equation*}
Choosing $l$ so that $l+\beta >0$, we have%
\begin{eqnarray*}
|J^{m}u_{\varepsilon }(x)|_{V}+|J^{m}u_{\varepsilon }|_{V,p}&\leq&
C\sum_{j=0}^{\infty }|J^{-l}u\ast \varphi _{j}|_{V,p} \leq
C\sum_{j=0}^{\infty }2^{-lj}|u\ast\varphi_j|_{V,p} \\
&\leq& C\sum_{j=0}^{\infty }2^{-(\beta+l)j}|u|_{B_{pp}^{\beta}(\mathbf{R}%
^d,V)}\leq C|u|_{B_{pp}^{\beta}(\mathbf{R}^d,V)}.
\end{eqnarray*}

For $u\in H_{p}^{\beta }(\mathbf{R}^{d},V),$ we have%
\begin{equation*}
J^{m}u_{\varepsilon }=J^{\beta }u\ast J^{m-\beta }\zeta_{\varepsilon } .
\end{equation*}%
Applying Minkowski's and H\"older's inequalities, we get 
\begin{equation*}
|J^{m}u_{\varepsilon }(x)|_{V}+|J^{m}u_{\varepsilon }|_{V,p}\leq C|J^{\beta
}u|_{V,p}= C|u|_{H_p^{\beta}(\mathbf{R}^d,V)}.
\end{equation*}
The lemma is proved.
\end{proof}

Similarly we approximate random functions.

\begin{lemma}
\label{cnew2}Let $\beta \in \mathbf{R},\ p\geq 1$ and $g\in \mathbb{A}$,
where $\mathbb{A}=\mathbb{B}_{pp}^{\beta }(E,V)$ or $\mathbb{H}_{p}^{\beta
}(E,V)$. Let 
\begin{equation}
\left( g\right) _{n}=g_{n}(t,x)=n\int_{t_{n}}^{t}g(s,\cdot )\ast \zeta _{%
\frac{1}{n}}(x)ds,\quad (t,x)\in E,  \label{eq14}
\end{equation}%
where $t_{n}=(t-\frac{1}{n})\vee 0$.

Then $|g_{n}-g|_{\mathbb{A}}\rightarrow 0$ as $n\rightarrow \infty $.
Moreover, for every $n$ and multiindex $\gamma \in \mathbf{N}_{0}^{d}$ there
is a constant $C$ not depending on $u$ such that 
\begin{eqnarray*}
&&\mathbf{E}\bigg[\sup_{(t,x)\in E}\big\vert\partial _{x}^{\gamma }g_{n}(t,x)%
\big\vert_{V}^{p}+\int_{0}^{T}\big\vert\partial _{x}^{\gamma }g_{n}(t,\cdot
)|_{V,p}^{p}dt\bigg] \\
&\leq &nC|g|_{\mathbb{A}}<\infty .
\end{eqnarray*}
\end{lemma}

\begin{proof}
Let 
\begin{equation*}
\tilde{g}_{n}(t,x)=n\int_{t_{n}}^{t}g(s,x)ds,\quad \bar{g}%
_{n}(t,x)=g(t,\cdot )\ast \zeta _{\frac{1}{n}}(x).
\end{equation*}%
Applying Minkowski's and H\"{o}lder's inequalities, we have 
\begin{eqnarray*}
|g_{n}-\tilde{g}_{n}|_{\mathbb{A}}^{p} &=&\int_{0}^{T}\bigg\vert %
n\int_{t_{n}}^{t}\big[\bar{g}_{n}(s,\cdot )-g(s,\cdot )\big]ds\bigg\vert_{%
\mathbb{A}_{d}}^{p}dt \\
&\leq &\int_{0}^{T}n\int_{t_{n}}^{t}\big\vert\bar{g}_{n}(s,\cdot )-g(s,\cdot
)\big\vert_{\mathbb{A}_{d}}^{p}dsdt \\
&\leq &\int_{0}^{T}\big\vert\bar{g}_{n}(s,\cdot )-g(s,\cdot )\big\vert_{%
\mathbb{A}_{d}}^{p}ds,
\end{eqnarray*}%
where $\mathbb{A}_{d}=\mathbb{B}_{pp}^{\beta }(\mathbf{R}^{d},V)$ or $%
\mathbb{H}_{p}^{\beta }(\mathbf{R}^{d},V)$. Hence, by Lemma~\ref{cnew}, $%
|g_{n}-\tilde{g}_{n}|_{\mathbb{A}}\rightarrow 0$ as $n\rightarrow \infty $.

Let $v(t),\ t\in[0,T]$, be a function in a Banach space with norm $||\cdot
|| $ such that $\int_0^T||v(t)||^pdt<\infty$. It is well known that 
\begin{equation*}
n\int_0^T\int_{t_n}^t||v(s)-v(t)||^pdsdt\to0
\end{equation*}
as $n\to\infty$. Therefore, applying Minkowski's and H\"older's
inequalities, we get 
\begin{eqnarray*}
|\tilde g_n-g|^p_{\mathbb{A}}&=&\int_{0}^{T}\bigg\vert n\int _{t_{n}}^t\big[%
g(s,\cdot)-g(t,\cdot)]ds\bigg\vert^{p}_{\mathbb{A}_d}dt \\
&\leq&\int_{0}^{T}\bigg[n\int_{t_n}^t\big\vert g(s,\cdot)-g(t,\cdot)%
\big\vert _{\mathbb{A}_d}ds\bigg]^pdt \\
&\leq& n\int_{0}^{T}\int_{t_n}^t\big\vert g(s,\cdot)-g(t,\cdot)\big\vert _{%
\mathbb{A}_d}^pdsdt\rightarrow 0
\end{eqnarray*}%
as $n\rightarrow \infty .$ Hence, 
\begin{equation*}
|g_{n}-g|_{\mathbb{A}}\leq |g_n-\tilde g_n|_{\mathbb{A}}+|\tilde g_n-g|_{%
\mathbb{A}}\to0,\quad n\to\infty.
\end{equation*}

Applying Minkowski's and H\"older's inequalities, we have 
\begin{eqnarray*}
|\partial_x^{\gamma} g_n(t,x)|^p_{V}&=&\bigg\vert n\int
_{t_{n}}^t\partial_x^{\gamma}\bar g_n(s,x)ds\bigg\vert^{p}_{V} \leq\bigg( %
n\int _{t_{n}}^t\big\vert\partial_x^{\gamma}\bar g_n(s,x)\big\vert_Vds\bigg)%
^{p} \\
&\leq& n\int _{t_{n}}^t\big\vert\partial_x^{\gamma}\bar g_n(s,x)\big\vert%
^p_Vds .
\end{eqnarray*}%
Therefore, by Lemma \ref{cnew}, 
\begin{eqnarray*}
&&\mathbf{E}\bigg(\sup_{(t,x)\in E} \big\vert\partial_x^{\gamma} g_n(t,x)%
\big\vert^{p}_{V}+\int_0^T \big\vert\partial_x^{\gamma} g_n(t,\cdot)ds%
\big\vert^{p}_{V,p}dt\bigg) \\
&&\qquad\leq n\mathbf{E}\bigg(\sup_{(t,x)\in E}\int_{t_n}^t \big\vert%
\partial_x^{\gamma}\bar g_n(s,x)ds\big\vert^{p}_{V}ds+\int_0^T \int_{t_n}^t%
\big\vert\partial_x^{\gamma}\bar g_n(s,\cdot)ds\big\vert^{p}_{V,p}dsdt\bigg)
\\
&&\qquad\leq nC|g|^p_{\mathbb{A}}<\infty.
\end{eqnarray*}%
The lemma is proved.
\end{proof}

We denote by $\mathfrak{D}_{p}(E,V)$, $p\geq 1$, the space of all $\mathcal{%
R(}\mathbb{F})\otimes \mathcal{B}(\mathbf{R}^{d})$-measurable $V$-valued
random functions $\Phi $ on $E$ such that $\Phi \in \cap _{\kappa >0}\mathbb{%
H}_{p}^{\kappa }(E,V)$ and for every multiindex 
\hbox{$\gamma \in
\mathbf{N}_{0}^{d}$}%
\begin{equation*}
\mathbf{E}\sup_{(t,x)\in E}|D_{x}^{\gamma }\Phi (t,x)|_{V}^{p}<\infty .
\end{equation*}%
Similarly we define the space $\mathfrak{D}_{p}(\mathbf{R}^{d},V)$ replacing 
$\mathcal{R}(\mathbb{F})$ and $E$ by $\mathcal{F}$ and $\mathbf{R}^{d}$ in
the definition of $\mathfrak{D}_{p}(E,V)$. For brevity of notation, if $%
V=L_{r}(U,\mathcal{U},\Pi ),\ r\geq 1,$ we write $\mathfrak{D}_{r,p}(E)=%
\mathfrak{D}_{p}(E,V).$ If $V=\mathbf{R}$, we drop $V$ in $\mathfrak{D}%
_{p}(E,V)$.

We denote by $\mathfrak{\bar{D}}_{r,p}(E),\,r\geq 1,\,p\geq 1$, the space of
all $\mathcal{R(}\mathbb{F})\otimes \mathcal{B}(\mathbf{R}^{d})\otimes 
\mathcal{B}(\mathbf{R}_{0}^{d})$ -measurable real-valued random functions $g$
such that for every multiindex $\gamma \in \mathbf{N}_{0}^{d}$ 
\begin{equation*}
\begin{array}{c}
\mathbf{E}\displaystyle\bigg\{\sup_{(t,x)\in E}\bigg[\int_{\mathbf{R}%
_{0}^{d}}|D_{x}^{\gamma }g(t,x,y)|^{r}l^{(\alpha )}(t,y)\frac{\displaystyle %
dy}{\displaystyle|y|^{d+\alpha }}\bigg]^{p/r} \\ 
\noalign{\vskip9pt} {\hspace*{.6cm}}+\displaystyle\int_{0}^{T}\int_{\mathbf{R%
}^{d}}\bigg[\int_{\mathbf{R}_{0}^{d}}|D_{x}^{\gamma }g(t,x,y)|^{r}l^{(\alpha
)}(t,y)\frac{\displaystyle dy}{\displaystyle|y|^{d+\alpha }}\bigg]^{p/r}dxdt%
\bigg\}<\infty .%
\end{array}%
\end{equation*}

Lemmas \ref{cnew} and \ref{cnew2} imply the following statement.

\begin{lemma}
\label{lemd} Let $p\geq1,\ r\geq1$ and $\kappa,\kappa^{\prime }\in \mathbf{R}
$. Then:

\emph{(a)} the set $\mathfrak{D}_{p}(\mathbf{R}^d)$ is a dense subset in $%
\mathbb{B}_{pp}^{\kappa }(\mathbf{R}^d)$ and the set $\mathfrak{D}_{p}(E)$
is a dense subset of $\mathbb{H}_{p}^{\kappa}(E)$;

\emph{(b)} the set $\mathfrak{D}_{r,p}(E)$ is a dense subset in $\mathbb{H}%
_{r,p}^{\kappa }(E)$ and $\mathbb{B}_{r,pp}^{\kappa }(E)$, and the set $%
\mathfrak{D}_{2,p}(E)\cap \mathfrak{D}_{p,p}(E)$ is a dense subset of $%
\mathbb{H}_{2,p}^{\kappa }(E)\cap \mathbb{B}_{p,pp}^{\kappa ^{\prime }}(E)$;

\emph{(c)} the set $\mathfrak{\bar{D}}_{r,p}(E)$ is a dense subset in $%
\mathbb{\bar{H}}_{r,p}^{\kappa }(E)$ and $\mathbb{\bar{B}}_{r,pp}^{\kappa
}(E)$, and the set $\mathfrak{\bar{D}}_{2,p}(E)\cap \mathfrak{\bar{D}}%
_{p,p}(E)$ is a dense subset of $\mathbb{\bar{H}}_{2,p}^{\kappa }(E)\cap 
\mathbb{\bar{B}}_{p,pp}^{\kappa ^{\prime }}(E)$.
\end{lemma}

\begin{proof}
(a) Let $u\in \mathbb{B}_{pp}^{\kappa }(\mathbf{R}^{d})$ and $u_{n}{(x)}%
=u\ast \zeta _{1/n}(x)$, $n=1,2,\ldots .$ Using Lemma \ref{cnew}, it is easy
to derive that $u_{n}\in \mathfrak{D}_{p}(\mathbf{R}^{d})$ and $|u-u_{n}|_{%
\mathbb{B}_{pp}^{\kappa }(\mathbf{R}^{d})}\rightarrow 0$ as $n\rightarrow
\infty $. If $g\in \mathbb{H}_{p}^{\kappa }(E)$ and $g_{n}$ is defined by (%
\ref{eq14}), then by Lemma~\ref{cnew2}, $|g-g_{n}|_{\mathbb{H}_{p}^{\kappa
}(E)}\rightarrow 0$ as $n\rightarrow \infty $.

(b)\enspace According to Lemma \ref{cnew2}, we have the following statements:

(i)\enspace if $g\in \mathbb{H}^{\kappa}_{r,p}(E)$ or $\mathbb{B}%
^{\kappa}_{r,pp}(E)$, then the functions $g_n$ defined by (\ref{eq14})
belong to $\mathfrak{D}_{r,p}(E)$ and $|g-g_n|_{\mathbb{H}%
^{\kappa}_{r,p}(E)}\to0$ or $|g-g_n|_{\mathbb{B}^{\kappa}_{r,pp}(E)}$ as $%
n\to\infty$;

(ii)\enspace if $g\in \mathbb{H}_{2,p}^{\kappa }(E)\cap \mathbb{B}%
_{p,pp}^{\kappa ^{\prime }}(E)$, then $g_{n}\in \mathfrak{D}_{2,p}(E)\cap 
\mathfrak{D}_{p,p}(E),\ n=1,\ldots ,$ and $|g-g_{n}|_{\mathbb{H}%
_{2,p}^{\kappa }(E)}+|g-g_{n}|_{\mathbb{B}_{p,pp}^{\kappa ^{\prime
}}(E)}\rightarrow 0$ as $n\rightarrow \infty $.

\noindent (c)\enspace According to Lemma 14.50 in \cite{Jac79}, there is a $%
\mathbf{R}^{d}$-valued $\mathcal{R}(\mathbb{F})\otimes \mathcal{B}(\mathbf{R}%
_{0})$-measurable function $c^{(\alpha )}(t,z)$ such that%
\begin{equation}
l^{(\alpha )}(t,y)\frac{dy}{|y|^{d+\alpha }}=\int_{\mathbf{R}%
_{0}}1_{dy}(c^{(\alpha )}(t,z))\frac{dz}{z^{2}}  \label{eq12}
\end{equation}%
and for any non-negative measurable function $F(t,x,y)$%
\begin{equation*}
\int_{\mathbf{R}_{0}^{d}}F(t,x,y)l^{(\alpha )}(t,y)\frac{dy}{|y|^{d+\alpha }}%
=\int_{\mathbf{R}_{0}}\tilde{F}(t,x,z)\frac{dz}{z^{2}},
\end{equation*}%
where $\tilde{F}(t,x,z)=F(t,x,c^{(\alpha )}(t,z))$. Hence, if $F\in \bar{%
\mathbb{H}}_{r,p}^{\kappa }(E)$, then $|F|_{\bar{\mathbb{H}}_{r,p}^{\kappa
}(E)}=|\tilde{F}|_{\mathbb{H}_{p}^{\kappa }(E,V_{r})}$, where $V_{r}=L_{r}{(%
\mathbf{R}_{0},\mathcal{B}(\mathbf{R}_{0}),dz/z^{2})}$. Also, if $F\in \bar{%
\mathbb{B}}_{r,pp}^{\kappa }(E)$, then $|F|_{\bar{\mathbb{B}}_{r,pp}^{\kappa
}(E)}=|\tilde{F}|_{\mathbb{B}_{pp}^{\kappa }(E,V_{r})}$.

Let $g\in \bar{\mathbb{H}}^{\kappa}_{r,p}(E)$ and $g_n$ be the function
defined by (\ref{eq14}). By Lemma \ref{cnew2}, $\tilde g_n\in \mathfrak{D}%
_{p}(E,V_r)$ and $|\tilde g-\tilde g_n|_{\mathbb{H}^{\kappa}_{p}(E,V_r)}\to0$
as $n\to\infty$. Therefore, $g_n\in \bar{\mathfrak{D}}_{r,p}(E)$ and $%
|g-g_n|_{\bar{\mathbb{H}}^{\kappa}_{r,p}(E)}=|\tilde g-\tilde g_n|_{{\mathbb{%
H}}^{\kappa}_{p}(E,V_r)} \to0$ as $n\to\infty$. So, $\bar{\mathfrak{D}}%
_{r,p}(E)$ is dense in $\bar{\mathbb{H}}^{\kappa}_{r,p}(E)$.

Similarly we prove the remaining assertions of part (c).
\end{proof}

\subsection{Stochastic integrals}

We discuss here the definition of the stochastic integrals with respect to a
martingale measure $\eta $.

\begin{lemma}
\label{ler1}Let $\beta \in \mathbf{R,}p\geq 2,\Phi \in \mathbb{H}%
_{2,p}^{\beta }(E)\cap \mathbb{H}_{p,p}^{\beta }(E)$. There is a unique
cadlag $H_{p}^{\beta }(\mathbf{R}^{d})$-valued process 
\begin{equation*}
M(t)=\int_{0}^{t}\int \Phi (s,x,\upsilon )\eta (ds,d\upsilon ),0\leq t\leq
T,x\in \mathbf{R}^{d},
\end{equation*}%
such that for every $\varphi \in \mathcal{S}(\mathbf{R}^{d})$%
\begin{equation}
\left\langle M(t),\varphi \right\rangle =\int_{0}^{t}\int \left\langle \Phi
(s,\cdot ,\upsilon ),\varphi \right\rangle \eta (ds,d\upsilon ),0\leq t\leq
T.  \label{00}
\end{equation}%
Moreover, there is a constant independent of $C$ such that%
\begin{equation*}
\mathbf{E}\sup_{t\leq T}|\int_{0}^{t}\int \Phi (s,\cdot ,\upsilon )\eta
(ds,d\upsilon )|_{H_{p}^{\beta }(\mathbf{R}^{d})}\leq
C\sum_{r=2,p}\left\vert \Phi \right\vert _{\mathbb{H}_{r,p}^{\beta }(E)}.
\end{equation*}
\end{lemma}

\begin{proof}
For an arbitrary $\phi \in \mathbb{H}_{2,p}^{\beta }(E)\cap \mathbb{H}%
_{p,p}^{\beta }(E)$, by stochastic Fubini theorem (Lemma 2 in \cite{MiPtoap})%
\begin{eqnarray*}
\int_{0}^{t}\int \left\langle \phi (s,\cdot ,\upsilon ),\varphi
\right\rangle \eta (ds,d\upsilon ) &=&\int_{0}^{t}\int \int J^{\beta }\phi
(s,x,\upsilon )J^{-\beta }\varphi (x)dx\eta (ds,d\upsilon ) \\
&=&\int \int_{0}^{t}\int J^{\beta }\phi (s,x,\upsilon )\eta (ds,d\upsilon
)J^{-\beta }\varphi (x)dx,
\end{eqnarray*}%
and (see Corollary 2 in \cite{mikprag1})%
\begin{eqnarray}
&&\mathbf{E[}\sup_{t\leq T}\left\vert \int_{0}^{t}\int \left\langle \phi
(s,\cdot ,\upsilon ),\varphi \right\rangle \eta (ds,d\upsilon )\right\vert
^{p}]  \label{01} \\
&\leq &C\int \mathbf{E}\sup_{t\leq T}\left\vert \int_{0}^{t}\int J^{\beta
}\phi (s,x,\upsilon )\eta (ds,d\upsilon )\right\vert ^{p}dx\leq
C\sum_{r=2,p}|\phi |_{\mathbb{H}_{r,p}^{\beta }(E)}.  \notag
\end{eqnarray}

First we define the stochastic integral for $\Phi \in \mathfrak{D}%
_{2,p}(E)\cap \mathfrak{D}_{p,p}(E)$. By Lemma 15 in \cite{MiPtoap}, for a
given $\Phi \in \mathfrak{D}_{2,p}(E)\cap \mathfrak{D}_{p,p}(E)$ there is a
cadlag in $t$ and smooth in $x$ adapted function $M(t,x)$ such that for each 
$\gamma \in \mathbf{N}_{0}^{d}$ and $x\in \mathbf{R}^{d}$, $\mathbf{P}$-a.s. 
\begin{equation*}
D_{x}^{\gamma }M(t,x)=\int_{0}^{t}\int D_{x}^{\gamma }\Phi (s,x,\upsilon
)\eta (ds,d\upsilon ),0\leq t\leq T.
\end{equation*}%
By stochastic Fubini theorem (Lemma 2 in \cite{MiPtoap}), for each $\beta
\in \mathbf{R}$ and $x\in \mathbf{R}^{d}$, $\mathbf{P}$-a.s.%
\begin{equation*}
J^{\beta }M(t,x)=\int_{0}^{t}\int J^{\beta }\Phi (s,x,\upsilon )\eta
(ds,d\upsilon ),0\leq t\leq T,
\end{equation*}%
and $\mathbf{P}$-a.s%
\begin{eqnarray}
\left\langle M(t),\varphi \right\rangle &=&\int M(t,x)\varphi
(x)dx=\int_{0}^{t}\int (\int \Phi (s,x,\upsilon )\varphi (x)dx)\eta
(ds,d\upsilon )  \label{for0} \\
&=&\int_{0}^{t}\int \left\langle \Phi (s,\cdot ,\upsilon ),\varphi
\right\rangle \eta (ds,d\upsilon ),0\leq t\leq T.  \notag
\end{eqnarray}%
Obviously, $J^{\beta }M(t)$ is $\mathbb{L}_{p}(\mathbf{R}^{d})$-valued
continuos and, by Corollary 2 in \cite{mikprag1}, there is a constant
independent of $\Phi $ such that%
\begin{equation}
\mathbf{E}\sup_{t\leq T}\left\vert M(t)\right\vert _{H_{p}^{\beta }(\mathbf{R%
}^{d})}^{p}\leq C\sum_{r=2,p}\left\vert \Phi \right\vert _{\mathbb{H}%
_{r,p}^{\beta }(E)}^{p}.  \label{for1}
\end{equation}%
If $\Phi \in \mathbb{H}_{2,p}^{\beta }(E)\cap \mathbb{H}_{p,p}^{\beta }(E)$,
then there is a sequence $\Phi _{n}\in \mathfrak{D}_{2,p}(E)\cap \mathfrak{D}%
_{p,p}(E)$ such that%
\begin{equation*}
\sum_{r=2,p}|\Phi -\Phi _{n}|_{\mathbb{H}_{r,p}^{\beta }(E)}\rightarrow 0
\end{equation*}%
as $n\rightarrow \infty $. Let%
\begin{equation*}
M_{n}(t)=\int_{0}^{t}\int \Phi _{n}(t,\cdot ,\upsilon )\eta (ds,d\upsilon
),0\leq t\leq T.
\end{equation*}%
According to (\ref{for1}) and (\ref{for0}),%
\begin{equation*}
\mathbf{E}\sup_{t\leq T}\left\vert M_{n}(t)-M_{m}(t)\right\vert
_{H_{p}^{\beta }(\mathbf{R}^{d})}^{p}\leq C\sum_{r=2,p}\left\vert \Phi
_{n}-\Phi _{m}\right\vert _{\mathbb{H}_{r,p}^{\beta }(E)}^{p}\rightarrow 0
\end{equation*}%
as $n,m\rightarrow \infty .$ Therefore there is an adapted cadlag $%
H_{p}^{\beta }(\mathbf{R}^{d})$-valued process $M(t)$ so that%
\begin{equation*}
\mathbf{E}\sup_{t\leq T}\left\vert M_{n}(t)-M(t)\right\vert _{H_{p}^{\beta }(%
\mathbf{R}^{d})}^{p}\rightarrow 0
\end{equation*}%
as $n\rightarrow \infty $. On the other hand by (\ref{01}), 
\begin{eqnarray*}
&&\mathbf{E}\sup_{t\leq T}|\int_{0}^{t}\int \left\langle \Phi _{n}(s,\cdot
,\upsilon )-\Phi (s,\cdot ,\upsilon ),\varphi \right\rangle \eta
(ds,d\upsilon )|^{p} \\
&\leq &C\sum_{r=2,p}|\Phi _{n}-\Phi |_{\mathbb{H}_{r,p}^{\beta
}(E)}\rightarrow 0
\end{eqnarray*}%
as $n\rightarrow \infty ,$ and (\ref{00}) holds. The statement follows.
\end{proof}

\begin{corollary}
\label{corr1}Let $\alpha \in (0,2),\ \beta \in \mathbf{R},\ p\geq 2,u\in 
\mathbb{H}_{p}^{\beta +\alpha }(E)$. Then%
\begin{equation*}
Q(t)=\int_{0}^{t}\int \left[ u(s,\cdot +y)-u(s,\cdot )\right] q^{(\alpha
)}(ds,dy),0\leq t\leq T,
\end{equation*}%
is cadlag $H_{p}^{\beta }(\mathbf{R}^{d})$-valued and%
\begin{equation*}
\mathbf{E}\sup_{t\leq T}|Q(t)|_{H_{p}^{\beta }(\mathbf{R}^{d})}\leq
C\left\vert u\right\vert _{\mathbb{H}_{p}^{\beta +\alpha /2}(E)}.
\end{equation*}
\end{corollary}

\begin{proof}
We apply Lemma \ref{ler1} with $\Phi (s,x,y)=u(s,x+y)-u(s,x),(s,x)\in E,y\in 
\mathbf{R}_{0}^{d}$. We have%
\begin{equation*}
J^{\beta }\Phi (s,x,y)=J^{\beta }u(s,x+y)-J^{\beta }u(s,x)
\end{equation*}%
and, by Theorem 2.2 in \cite{stri},%
\begin{eqnarray*}
&&\int \left( \int J^{\beta }\Phi (s,x,y)|^{2}\frac{dy}{|y|^{d+\alpha }}%
\right) ^{p/2}dx \\
&=&\int \left( \int |J^{\beta }u(s,x+y)-J^{\beta }u(s,x)|^{2}\frac{dy}{%
|y|^{d+\alpha }}\right) ^{p/2} \\
&\leq &C\left\vert u\right\vert _{\mathbb{H}_{p}^{\beta +\alpha /2}}^{p}.
\end{eqnarray*}%
By definition of the norm,%
\begin{eqnarray*}
&&\mathbf{E}\int_{0}^{T}\int \int J^{\beta }\Phi (s,x,y)|^{p}\frac{dydxds}{%
|y|^{d+\alpha }} \\
&=&\mathbf{E}\int_{0}^{T}\int \int |J^{\beta }u(s,x+y)-J^{\beta }u(s,x)|^{p}%
\frac{dydxds}{|y|^{d+\alpha }} \\
&\leq &C\left\vert u\right\vert _{\mathbb{B}_{pp}^{\beta +\alpha
/p}}^{p}\leq C\left\vert u\right\vert _{\mathbb{H}_{p}^{\beta +\alpha
/2}}^{p}.
\end{eqnarray*}
\end{proof}

We will need the following auxiliary statement as well. Let $(\overline{%
\Omega },\overline{\mathcal{F}},\overline{\mathbf{P}})$ be a complete
probability space. We introduce the product of probability spaces 
\begin{equation*}
(\widetilde{\Omega },\widetilde{\mathcal{F}},\widetilde{\mathbf{P}})=(\Omega
\times \overline{\Omega },\mathcal{F}\otimes \overline{\mathcal{F}},\mathbf{P%
}\times \overline{\mathbf{P}}).
\end{equation*}%
Let $\widetilde{\mathbb{F}}^{\prime }=(\widetilde{\mathcal{F}}_{t}^{\prime
}) $ be the usual augmentation of $(\mathcal{F}_{t}\otimes \overline{%
\mathcal{F}})$ (see \cite{delmey}). Obviously, $\eta (dt,d\upsilon )$ is a $(%
\widetilde{\mathbb{F}}^{\prime },\mathbf{\tilde{P}})$-martingale measure.
For a measurable integrable function $F$ on $\tilde{\Omega}=\Omega \times 
\bar{\Omega}$ we denote%
\begin{eqnarray*}
\mathbf{E}^{(1)}F &=&\int F(\omega ,\bar{\omega})\mathbf{P}(d\omega ),\omega
\in \Omega , \\
\mathbf{E}^{(2)}F &=&\int F(\omega ,\bar{\omega})\mathbf{\bar{P}}(d\bar{%
\omega}),\bar{\omega}\in \bar{\Omega}.
\end{eqnarray*}

\begin{lemma}
\label{lep1}(a) Let $\xi (s,\omega ,\bar{\omega},\upsilon )$ be $\mathcal{R}(%
\widetilde{\mathbb{F}}^{\prime })\otimes \mathcal{U}$-measurable and 
\begin{equation*}
\mathbf{E}^{\mathbf{\tilde{P}}}\int_{0}^{T}\int_{U}|\xi (s,\upsilon
)|^{2}\Pi (d\upsilon )ds<\infty .
\end{equation*}%
Then $\mathbf{\tilde{P}}$-a.s. for all $t\geq 0,$%
\begin{eqnarray*}
\mathbf{E}^{(1)}\int_{0}^{t}\int \xi (s,\upsilon )\eta (ds,d\upsilon ) &=&0,
\\
\mathbf{E}^{(2)}\int_{0}^{t}\int \xi (s,\upsilon )\eta (ds,d\upsilon )
&=&\int_{0}^{t}\int \mathbf{E}^{(2)}\xi (s,\upsilon )\eta (ds,d\upsilon ).
\end{eqnarray*}

(b) Let $\xi (s,\omega ,\bar{\omega})$ be $Y$-valued $\mathcal{R}(\widetilde{%
\mathbb{F}}^{\prime })$-measurable and 
\begin{equation*}
\mathbf{E}^{\mathbf{\tilde{P}}}\int_{0}^{T}|\xi (s)|_{Y}^{2}ds<\infty .
\end{equation*}%
Then $\mathbf{\tilde{P}}$-a.s. for all $t\geq 0,$%
\begin{eqnarray*}
\mathbf{E}^{(1)}\int_{0}^{t}\xi (s)dW_{s} &=&0, \\
\mathbf{E}^{(2)}\int_{0}^{t}\xi (s)dW_{s} &=&\int_{0}^{t}\int \mathbf{E}%
^{(2)}\xi (s)dW_{s}.
\end{eqnarray*}
\end{lemma}

\begin{proof}
(a) Obviously,%
\begin{equation*}
M_{t}=\int_{0}^{t}\int \xi (s,\upsilon )\eta (ds,d\upsilon ),0\leq t\leq T,
\end{equation*}%
is a $(\widetilde{\mathbb{F}}^{\prime },\mathbf{\tilde{P}})$-martingale.
Then for any $A\in \mathcal{\bar{F}}$,%
\begin{equation*}
0=\int \chi _{A}(\bar{\omega})M_{t}(\omega ,\bar{\omega})\mathbf{\tilde{P}}%
(d\omega ,d\bar{\omega})=\int \chi _{A}(\bar{\omega})\mathbf{E}%
^{(1)}M_{t}(\cdot ,\bar{\omega})\mathbf{\bar{P}}(d\bar{\omega}).
\end{equation*}%
Since $A\in \mathcal{\bar{F}}$ is arbitrary, it follows that $\mathbf{E}%
^{(1)}M_{t}(\bar{\omega})=0$ $\mathbf{\bar{P}}$-a.s.

Obviously, there is a sequence of the form%
\begin{equation*}
\xi _{n}(s,\omega ,\bar{\omega},\upsilon )=\sum_{k=1}^{N_{n}}\phi _{n,k}(%
\bar{\omega})\xi _{n,k}(s,\omega ,\upsilon )
\end{equation*}%
such that $\phi _{n,k}$ are $\mathcal{\bar{F}}$- and $\xi _{n,k}$ are $%
\mathcal{P}(\mathbb{F})\otimes \mathcal{U}$-measurable,%
\begin{equation*}
\mathbf{E}^{\mathbf{\tilde{P}}}\int_{0}^{T}\int \phi _{n,k}(\bar{\omega}%
)^{2}\xi _{n,k}(s,\omega ,\upsilon )^{2}\Pi (d\upsilon )ds<\infty ,n\geq
1,1\leq k\leq N_{n},
\end{equation*}%
and 
\begin{equation*}
\mathbf{E}^{\mathbf{\tilde{P}}}\int_{0}^{T}\int (\xi _{n}(s,\omega ,,\bar{%
\omega},\upsilon )-\xi _{n}(s,\omega ,,\bar{\omega},\upsilon ))^{2}\Pi
(d\upsilon )ds\rightarrow 0
\end{equation*}%
as $n\rightarrow \infty $. Since for any $k,n$%
\begin{eqnarray*}
&&\mathbf{E}^{(2)}\int_{0}^{t}\int \phi _{n,k}\xi _{n,k}(s,\upsilon )\eta
(ds,d\upsilon ) \\
&=&\mathbf{E}^{\mathbf{\bar{P}}}\phi _{n,k}(\bar{\omega})\int_{0}^{t}\int
\xi _{n,k}(s,\upsilon )\eta (ds,d\upsilon )=\int_{0}^{t}\int \mathbf{E}^{%
\mathbf{\bar{P}}}\phi _{n,k}(\bar{\omega})\xi _{n,k}(s,\upsilon )\eta
(ds,d\upsilon ) \\
&=&\int_{0}^{t}\int \mathbf{E}^{(2)}[\phi _{n,k}\xi _{n,k}(s,\upsilon )]\eta
(ds,d\upsilon ),
\end{eqnarray*}%
we have%
\begin{equation*}
\mathbf{E}^{(2)}\int_{0}^{t}\int \xi _{n}(s,\upsilon )\eta (ds,d\upsilon
)=\int_{0}^{t}\int \mathbf{E}^{(2)}\xi _{n}(s,\upsilon )\eta (ds,d\upsilon )
\end{equation*}%
and the statement follows by passing to the limit.

(b) The satement is shown by repeating with obvious changes the proof of the
part (a).
\end{proof}

\section{Model problem. Partial case I}

In this section, we consider the Cauchy problem%
\begin{eqnarray}
du(t,x) &=&\big(A_{0}^{(\alpha )}u-\lambda u+f\big)(t,x)dt  \label{eq15} \\
&&+\int_{U}\Phi (t,x,\upsilon )\eta (dt,d\upsilon )+h(t,x)dW_{t}%
\mbox{\quad
in }E,  \notag \\
u(0,x) &=&u_{0}(x)\mbox{\quad in }\mathbf{R}^{d}  \notag
\end{eqnarray}%
for smooth in $x$ input functions $u_{0},f,\Phi ,h$.

The operator $A_0^{(\alpha)},\ \alpha\in(0,2]$, is defined by%
\begin{eqnarray*}
A_{0}^{(\alpha )}u(t,x) &=&\int \nabla _{y}^{\alpha }u(t,x)m_{0}^{(\alpha
)}(t,y)\frac{dy}{|y|^{d+\alpha }}1_{\alpha\in(0,2)} \\
&&+(b(t),\nabla u(t,x))1_{\alpha =1}+\frac12{\delta }\Delta u(t,x)1_{\alpha
=2},
\end{eqnarray*}
where $\delta>0$, 
\begin{equation*}
\nabla _{y}^{\alpha }u(t,x) = u(t,x+y)-u(t,x)-(\nabla u(t,x),y)\chi
^{(\alpha )}{(y)},
\end{equation*}%
$\chi ^{(\alpha )}{(y)}=1_{\alpha\in(1,2)}+1_{|y|\leq1}1_{\alpha=1}$ and $%
\Delta$ is the Laplace operator in $\mathbf{R}^d$.

\subsection{Auxiliary results}

\bigskip In terms of Fourier transform%
\begin{equation*}
\mathcal{F}(A_{0}^{(\alpha )}u)(t,\xi )=\psi _{0}^{(\alpha )}(t,\xi )%
\mathcal{F}u(t,\xi ),
\end{equation*}%
where%
\begin{eqnarray*}
\psi _{0}^{(\alpha )}(t,\xi ) &=&-c_{0}\int_{S^{d-1}}|(w,\xi )|^{\alpha }%
\Big[1-i\Big(\tan \frac{\alpha \pi }{2}\,\mbox{sgn}(w,\xi )1_{\alpha \neq 1}
\\
&&-\frac{2}{\pi }\,\mbox{sgn}(w,\xi )\ln |(w,\xi )|1_{\alpha =1}\Big)\Big]%
m_{0}^{(\alpha )}(t,w)\mu _{d-1}(dw)1_{\alpha \in (1,2)} \\
&&+i(b(t),\xi )1_{\alpha =1}-\frac{1}{2}{\delta }|\xi |^{2}1_{\alpha =2}.
\end{eqnarray*}%
and $c_{0}=c_{0}(\alpha )$ is a positive constant.

Let us introduce the functions%
\begin{eqnarray*}
G_{s,t}^{(\alpha)}(x ) &=&\mathcal{F}^{-1}\bigg\{\exp \bigg( %
\int_{s}^{t}\psi _{0}^{(\alpha )}(r,\xi )dr\bigg)\bigg\}(x), \\
G_{s,t}^{(\alpha),\lambda}(x) &=&e^{-\lambda(t-s)}G^{(\alpha)
}_{s,t}(x)\quad 0\leq s<t\leq T,\ x\in\mathbf{R}^d.
\end{eqnarray*}

\begin{remark}
\label{re2} \emph{The function $G_{s,t}^{(\alpha )}(\alpha )(x)$ is the
fundamental solution of the equation $\partial _{t}u=A_{0}^{(\alpha )}u$. On
the other hand (see, e.g. \cite{SaT94}), $G_{s,t}^{(\alpha )}$ is the
density function of an $\alpha $-stable distribution. Hence, $%
G_{s,t}^{(\alpha )}\geq 0$ and} 
\begin{equation*}
\int_{\mathbf{R}^{d}}G_{s,t}^{(\alpha )}(x)dx=1.
\end{equation*}
\end{remark}

Further, for brevity of notation, we will drop the superscript $\alpha$ in $%
G_{s,t}^{(\alpha),\lambda}$ and $G_{s,t}^{(\alpha)}$.

For a representation of solution to (\ref{eq15}) we introduce the following
operators:%
\begin{eqnarray*}
T_{t}^{\lambda }u_{0}(x) &=&G_{0,t}^{\lambda }\ast u_{0}(x),\quad u_{0}\in 
\mathfrak{D}_{p}(\mathbf{R}^{d}), \\
R_{\lambda }f(t,x) &=&\int_{0}^{t}G_{s,t}^{\lambda }\ast f(s,x)ds,\quad f\in 
\mathfrak{D}_{p}(E), \\
\widetilde{R}_{\lambda }\Phi (t,x) &=&\int_{0}^{t}\int_{U}G_{s,t}^{\lambda
}\ast \Phi (s,x,\upsilon )\eta (ds,d\upsilon ),\quad \Phi \in \mathfrak{D}%
_{2,p}(E)\cap \mathfrak{D}_{p,p}(E), \\
\bar{R}_{\lambda }h(t,x) &=&\int_{0}^{t}G_{s,t}^{\lambda }\ast
h(s,x)dW_{s},\quad h\in \mathfrak{D}_{p}(E,Y).
\end{eqnarray*}

\begin{lemma}
\label{le1}Let $\alpha \in (0,2],\ p\geq 2$ and Assumption \emph{A$_{0}$} be
satisfied. Then there is a constant $C=C(\alpha ,p,d,K,\delta )$ such that
the following estimates hold:%
\begin{eqnarray*}
(\mathrm{i})\quad &&|T^{\lambda }u_{0}|_{\mathbb{L}_{p}(E)}\leq \rho
_{\lambda }^{\frac{1}{p}}|u_{0}|_{\mathbb{L}_{p}(\mathbf{R}^{d})},\quad
u_{0}\in \mathfrak{D}_{p}(\mathbf{R}^{d}), \\
(\mathrm{ii})\quad &&|{R}_{\lambda }f|_{\mathbb{L}_{p}(E)}\leq \rho
_{\lambda }|f|_{\mathbb{L}_{p}(E)},\quad f\in \mathfrak{D}_{p}(E), \\
(\mathrm{iii})\quad &&|\widetilde{R}_{\lambda }\Phi |_{\mathbb{L}%
_{p}(E)}\leq C\sum_{r=2,p}\rho _{\lambda }^{\frac{1}{r}}|\Phi |_{\mathbb{L}%
_{r,p}(E)},\quad \Phi \in \mathfrak{D}_{2,p}(E)\cap \mathfrak{D}_{p,p}(E), \\
(\mathrm{iv})\quad &&|\bar{R}_{\lambda }h|_{\mathbb{L}_{p}(E)}\leq C\rho
_{\lambda }^{\frac{1}{2}}|h|_{\mathbb{L}_{p}(E,Y)},\quad h\in \mathfrak{D}%
_{p}(E,Y),
\end{eqnarray*}%
where $\rho _{\lambda }=T\wedge \frac{1}{\lambda }$.
\end{lemma}

\begin{proof}
(i)\enspace By Minkowski's inequality and Remark \ref{re2}, 
\begin{eqnarray*}
|T^{\lambda }u_0|_{\mathbb{L}_{p}(E)}^{p} &=&\mathbf{E}\int_{0}^{T}\int_{%
\mathbf{R}^{d}} |G_{0,t}^{\lambda}\ast u_{0}(x)|^{p}dxdt \leq \int_{0}^{T}
e^{-\lambda pt}|u_{0}|^{p}_{\mathbb{L}_p(\mathbf{R}^d)}dt \\
&\leq&\rho_{\lambda}|u_{0}|_{\mathbb{L}_{p}(\mathbf{R}^{d})}^{p}.
\end{eqnarray*}%
(ii)\enspace By Minkowski's inequality and Remark \ref{re2}, 
\begin{eqnarray*}
|R_{\lambda }f|_{\mathbb{L}_{p}(E)}^{p} &=&\mathbf{E}\int_{0}^{T}\int_{%
\mathbf{R}^{d}}\bigg\vert\int_{0}^{t} G^{\lambda}_{s,t}\ast f(s,x)ds%
\bigg\vert^{p}dxdt \\
&\leq &\int_{0}^{T}\bigg(\int_{0}^{t}\int_{\mathbf{R}^d}G^{%
\lambda}_{s,t}(y)|f(s,\cdot)|_{\mathbb{L}_p(\mathbf{R}^d)}dyds\bigg)^{p}dt \\
&=&\int_{0}^{T}\bigg(\int_{0}^{t}e^{-\lambda(t-s)}|f(s,\cdot)|_{\mathbb{L}_p(%
\mathbf{R}^d)}ds\bigg)^{p}dt \equiv I.
\end{eqnarray*}

Applying here H\"older's inequality, we get 
\begin{eqnarray}
I&\leq &\int_{0}^{T}\bigg(\int_{0}^{t}e^{-\lambda(t-s)}ds\bigg)%
^{p-1}\int_{0}^{t}e^{-\lambda(t-s)}|f(s, \cdot)|^p_{\mathbb{L}_p(\mathbf{R}%
^d)}dsdt  \notag \\
&\leq & \rho^{p-1}_{\lambda}\int_{0}^{T}\int_{s}^{T}e^{-\lambda(t-s)}|f(s,
\cdot)|^p_{\mathbb{L}_p(\mathbf{R}^d)}dtds \leq \rho^p_{\lambda}|f|^p_{%
\mathbb{L}_p(E)}.  \label{eq16}
\end{eqnarray}%
(iii)\enspace By Doob's and Minkowski's inequalities%
\begin{eqnarray*}
\mathbf{E}|\bar{R}_{\lambda }h(t,x) |^{p} &=&\mathbf{E}\bigg\vert%
\int_{0}^{t}G^{\lambda}_{s,t}\ast h (s,x)dW_s\bigg\vert^p \leq C\mathbf{E}%
\bigg(\int_{0}^{t}\big\vert G^{\lambda}_{s,t}\ast h (s,x)\big\vert^2_Yds%
\bigg)^{\frac{p}{2}} \\
&\leq& C\mathbf{E}\bigg(\int_{0}^{t}\big[ G^{\lambda}_{s,t}\ast |h (s,x)|_Y%
\big]^2ds\bigg)^{\frac{p}{2}} .
\end{eqnarray*}%
Therefore, by Minkowski's inequality and Remark~\ref{re2},%
\begin{eqnarray*}
|\bar{R}_{\lambda }h |^{p}_{\mathbb{L}_p(E)} &=&\mathbf{E}\int_{0}^{T}\int_{%
\mathbf{R}^d}\big\vert\bar R_{\lambda}h(t,x)\big\vert^p dxdt \\
&\leq& C\mathbf{E}\int_{0}^{T}\int_{\mathbf{R}^d}\bigg(\int_0^t\big[ %
G^{\lambda}_{s,t}\ast |h (s,x)|_Y\big]^2ds\bigg)^{\frac{p}{2}}dxdt \\
&\leq& C\int_0^T\bigg(\int_{0}^{t}\int_{\mathbf{R}^d}e^{-2\lambda(t-s)}
G_{s,t}(y)|h (s,\cdot)|^2_{\mathbb{L}_p(\mathbf{R}^d,Y)}dyds\bigg)^{\frac{p}{%
2}}dt \\
&=& C\int_0^T\bigg(\int_{0}^{t}e^{-2\lambda(t-s)}|h (s,\cdot)|^2_{\mathbb{L}%
_p(\mathbf{R}^d,Y)}ds\bigg)^{\frac{p}{2}}dt .
\end{eqnarray*}%
Now, similarly as in (\ref{eq16}) with $p$ replaced by $p/2$, we get%
\begin{equation*}
|\bar R_{\lambda}h|^p_{\mathbb{L}_p(E)}\leq C \rho_{\lambda}^{\frac{p}{2}%
}|h|^p_{\mathbb{L}_{p}(E)} .
\end{equation*}

By Corollary 2 in \cite{mikprag1}, 
\begin{equation*}
|\tilde{R}_{\lambda }\Phi |_{\mathbb{L}_{p}(E)}^{p}\leq C(A+B),
\end{equation*}%
where 
\begin{equation*}
A=\mathbf{E}\int_{0}^{T}\int_{\mathbf{R}^{d}}\bigg(\int_{0}^{t}\int_{U}\big[%
G_{s,t}^{\lambda }\ast \Phi (s,x,v)\big]^{2}\Pi (dv)ds\bigg)^{\frac{p}{2}%
}dxdt
\end{equation*}%
and 
\begin{equation*}
B=\mathbf{E}\int_{0}^{T}\int_{\mathbf{R}^{d}}\int_{0}^{t}\int_{U}\big\vert %
G_{s,t}^{\lambda }\ast \Phi (s,x,v)\big\vert^{p}\Pi (dv)dsdxdt.
\end{equation*}

By Minkowski's inequality and Remark~\ref{re2},%
\begin{eqnarray*}
A&\leq&\mathbf{E}\int_0^T\int_{\mathbf{R}^d}\bigg( \int_0^t \bigg[ %
G^{\lambda}_{s,t}\ast\bigg(\int_U\Phi^2(s,x,v)\Pi(dv)\bigg)^{\frac{1}{2}}%
\bigg]^2ds\bigg)^{\frac{p}{2}}dxdt \\
&\leq&\int_0^T\bigg( \int_0^t e^{-2\lambda(t-s)}|\Phi(s,\cdot,\cdot)|^{2}_{%
\mathbb{L}_{2,p}(\mathbf{R}^d)}ds\bigg)^{\frac{p}{2}}dt
\end{eqnarray*}%
and 
\begin{eqnarray*}
B&\leq&\mathbf{E}\int_0^T\int_{\mathbf{R}^d}\int_0^t \bigg[ %
G^{\lambda}_{s,t}\ast\bigg(\int_U|\Phi(s,x,v)|^p\Pi(dv)\bigg)^{\frac{1}{p}}%
\bigg]^pdsdxdt \\
&\leq&\int_0^T\int_0^t e^{-p\lambda(t-s)}|\Phi(s,\cdot,\cdot)|^{p}_{\mathbb{L%
}_{p,p}(\mathbf{R}^d)}dsdt.
\end{eqnarray*}%
Now, similarly as in (\ref{eq16}), we get 
\begin{equation*}
A\leq\rho_{\lambda}^{\frac{p}{2}}|\Phi|^p_{\mathbb{L}_{2,p}(E)}
\end{equation*}%
and 
\begin{equation*}
B\leq\rho_{\lambda}|\Phi|^p_{\mathbb{L}_{p,p}(E)}.
\end{equation*}
\end{proof}

\begin{lemma}
\label{le2}Let $\alpha \in (0,2],\ \beta \in \mathbf{R},\ p\geq 2$ and
Assumption $\mathbf{A_{0}}$ be satisfied. Then there is a constant $%
C=C(\alpha ,p,d,K,\delta ,T)$ such that the following estimates hold:%
\begin{eqnarray*}
(\mathrm{i})\quad &&|T^{\lambda }u_{0}|_{\mathbb{H}_{p}^{\beta +\alpha
}(E)}\leq C|u_{0}|_{\mathbb{B}_{pp}^{\beta +\alpha -\frac{\alpha }{p}}(%
\mathbf{R}^{d})},\ u_{0}\in \mathfrak{D}_{p}(\mathbf{R}^{d}), \\
(\mathrm{ii})\quad &&|R_{\lambda }f|_{\mathbb{H}_{p}^{\beta +\alpha
}(E)}\leq C|f|_{\mathbb{H}_{p}^{\beta }(E)},\ f\in \mathfrak{D}_{p}(E), \\
(\mathrm{iii})\quad &&|\widetilde{R}_{\lambda }\Phi |_{\mathbb{H}_{p}^{\beta
+\alpha }(E)}\leq C\sum_{r=2,p}|\Phi |_{\mathbb{B}_{r,pp}^{\beta +\alpha -%
\frac{\alpha }{r}}(E)},\ \Phi \in \mathfrak{D}_{2,p}(E)\cap \mathfrak{D}%
_{p,p}(E), \\
(\mathrm{iv})\quad &&|\bar{R}_{\lambda }h|_{\mathbb{H}_{p}^{\beta +\alpha
}(E)}\leq C|h|_{\mathbb{H}_{p}^{\beta +\alpha /2}(E,Y)},\ h\in \mathfrak{D}%
_{p}(E,Y).
\end{eqnarray*}
\end{lemma}

\begin{proof}
We have $\mathbb{L}_{p}$-estimates by Lemma \ref{le1}. The estimate (ii)
follows by Theorem 2.1 in \cite{MiP922}. The estimate (iii) is proved in 
\cite{mikprag1} (we apply Corollary 1 and Proposition~2 with $V=L_{2}(U,%
\mathcal{U},\Pi )$). The estimate (iv) is proved in \cite{mikprag1}
(Proposition 2 with $V=Y$).

It remains to prove (i). We follow the arguments in \cite{mikprag1} (see 
\cite{cl} as well). Let 
\begin{equation*}
T_{t}u_{0}(x)=G_{0,t}\ast u_{0}(x).
\end{equation*}%
Since 
\begin{equation*}
|T_{\cdot }^{\lambda }u_{0}|_{\mathbb{H}_{p}^{\beta }(E)}^{p}=\mathbf{E}%
\int_{0}^{T}e^{-\lambda pt}|T_{t}u_{0}|_{H_{p}^{\beta }{(\mathbf{R}^{d})}%
}^{p}dt\leq \mathbf{E}|T_{\cdot }u_{0}|_{H_{p}^{\beta }{(E)}}^{p},
\end{equation*}%
it suffices to prove the estimate 
\begin{equation}
|T_{\cdot }u_{0}|_{H_{p}^{\beta }(E)}\leq C|u_{0}|_{B_{pp}^{\beta -\frac{%
\alpha }{p}}(\mathbf{R}^{d})}  \label{eq23}
\end{equation}%
for non-random functions $u_{0}\in \mathfrak{D}_{p}{(\mathbf{R}^{d})}$.

In order to show (\ref{eq23}), we follow \cite{MiP922}. Let 
\begin{eqnarray*}
\widetilde{\varphi }_{0} &=&{}\varphi _{0}+\varphi _{1}, \\
\widetilde{\varphi }_{j} &=&{}\varphi _{j-1}+\varphi _{j}+\varphi
_{j+1},\quad j\geqslant 1,
\end{eqnarray*}%
where the functions $\varphi _{j},\ j\geq 0,$ are defined in Subsection \ref%
{test}. Let%
\begin{equation*}
h_{t}^{j}(x)=G_{0,t}\ast \widetilde{\varphi }_{j}(x),\quad j\geqslant 0.
\end{equation*}%
According to Lemma 12 in \cite{MiP922} or inequality (36) and Lemma 16 in 
\cite{mikprag1}, there are positive constants $C$ and $c$ such that for all $%
s<t,\ j\geq 1,$%
\begin{equation}
\big\vert h_{t}^{j}\big\vert_{1}\leq Ce^{-c2^{j\alpha }t}\sum_{k\leq d_{0}}%
\big(2^{j\alpha }t\big)^{k},\quad |h_{t}^{0}|_{1}\leq C.  \label{ff1}
\end{equation}%
Here and in the remaining part of the proof we use the notation $|\cdot
|_{p}=|\cdot |_{L_{p}(\mathbf{R}^{d})},\ p\geq 1$.

We set%
\begin{equation*}
u_{0,j}=u_{0}\ast \varphi _{j},\quad j\geqslant 0.
\end{equation*}%
Obviously, 
\begin{equation*}
{}\varphi _{j}\ast T_{t}u_{0}=T_{t}(u_{0}\ast \varphi
_{j})=T_{t}u_{0,j},\quad j\geqslant 0.
\end{equation*}%
Since $\varphi _{j}=\varphi _{j}\ast \widetilde{\varphi }_{j},\quad j\geq 0$%
, we have 
\begin{equation*}
T_{t}u_{0,j}=h_{t}^{j}\ast u_{0,j},\quad j\geq 0.
\end{equation*}%
By Minkowski's inequality, 
\begin{eqnarray}
|T_{\cdot }u_{0}|_{H_{p}^{\beta }(E)}^{p} &=&\int_{0}^{T}\int_{\mathbf{R}%
^{d}}\bigg(\sum_{j=0}^{\infty }2^{2\beta j}\big[\varphi _{j}\ast
T_{t}u_{0}(x)\big]^{2}\bigg)^{\frac{p}{2}}dxdt  \notag \\
&=&\int_{0}^{T}\int_{\mathbf{R}^{d}}\bigg(\sum_{j=0}^{\infty }2^{2\beta j}%
\big[h_{t}^{j}\ast u_{0,j}(x)\big]^{2}\bigg)^{\frac{p}{2}}dxdt  \notag \\
&\leq &\int_{0}^{T}\bigg(\sum_{j=0}^{\infty }2^{2\beta j}\big\vert %
h_{t}^{j}\ast u_{0,j}\big\vert_{p}^{2}\bigg)^{\frac{p}{2}}dt.  \label{ff20}
\end{eqnarray}%
Applying Minkowski's inequality and (\ref{ff1}), we get%
\begin{equation*}
|h_{t}^{j}\ast u_{0,j}|_{p}\leq |h_{t}^{j}|_{1}\,|u_{0,j}|_{p}\leq
Ce^{-c2^{\alpha j}t}|u_{0,j}|_{p},\quad j\geq 0.
\end{equation*}%
Hence, by (\ref{ff20}) 
\begin{equation}
|T_{\cdot }u_{0}|_{H_{p}^{\beta }(E)}^{p}\leq C\int_{0}^{T}\bigg(%
\sum_{j=0}^{\infty }e^{-c2^{\alpha j}t}2^{2\beta j}|u_{0,j}|_{p}^{2}\bigg)^{%
\frac{p}{2}}dt.  \label{eq18}
\end{equation}%
If $p=2$, we have immediately 
\begin{eqnarray*}
|T_{\cdot }u_{0}|_{H_{p}^{\beta }(E)}^{p} &\leq
&C\int_{0}^{T}\sum_{j=0}^{\infty }e^{-c2^{\alpha j}t}2^{2\beta
j}|u_{0,j}|_{2}^{2}dt \\
&\leq &C\sum_{j=0}^{\infty }2^{-{\alpha j}}2^{2\beta
j}|u_{0,j}|_{2}^{2}=C|u_{0}|_{B_{22}^{\beta -\frac{\alpha }{2}}(\mathbf{R}%
^{d})}.
\end{eqnarray*}%
If $p>2$, we split the sum in (\ref{eq18}) as follows: 
\begin{eqnarray*}
&&\sum_{j=0}^{\infty }e^{-c2^{\alpha j}t}2^{2\beta
j}|u_{0,j}|_{p}^{2}=\sum_{j\in J}e^{-c2^{\alpha j}t}2^{2\beta
j}|u_{0,j}|_{p}^{2} \\
&&\qquad +\sum_{j\in \mathbb{N}_{0}\setminus J}e^{-c2^{\alpha j}t}2^{2\beta
j}|u_{0,j}|_{p}^{2}=A(t)+B(t),
\end{eqnarray*}%
where $J=\{j\in \mathbb{N}_{0}\colon \ 2^{\alpha j}t\leq 1\}$.

Fix $\kappa \in (0,\frac{2\alpha }{p}).$ Using H\"{o}lder's inequality, we
get%
\begin{eqnarray*}
A(t) &\leq &\sum_{j\in J}2^{2\beta j}2^{\kappa j}2^{-\kappa
j}|u_{0,j}|_{p}^{2} \leq \bigg( \sum_{j\in J}2^{q\kappa j}\bigg) ^{1/q}%
\bigg( \sum_{j\in J}2^{p\beta j}2^{-p\kappa j/2}|u_{0,j}|_{p}^{p}\bigg) %
^{2/p}
\end{eqnarray*}%
with $q=\frac{p}{p-2}$. Since%
\begin{equation*}
\sum_{j\in J}2^{q\kappa j} \leq Ct^{-q\kappa/\alpha },
\end{equation*}%
we have%
\begin{eqnarray*}
A(t) &\leq &Ct^{-\frac{\kappa }{\alpha }}\bigg( \sum_{j\in J}2^{p\beta
j}2^{-p\kappa j/2}|u_{0,j}|_{p}^{p}\bigg) ^{\frac{2}{p}} \\
&=&Ct^{-\frac{\kappa }{\alpha }}\bigg( \sum_{j}1_{\{ t\leq 2^{-\alpha j}\}
}2^{p\beta j}2^{-p\kappa j/2}|u_{0,j}|_{p}^{p}\bigg) ^{\frac{2}{p}}.
\end{eqnarray*}%
So,%
\begin{eqnarray*}
\int_{0}^{T}A(t)^{\frac{p}{2}}dt &\leq &C\sum_{j}2^{p\beta j}2^{-p\kappa
j/2}|u_{0,j}|_{p}^{p}\int_{0}^{2^{-\alpha j}}t^{-\frac{p\kappa }{2\alpha }}dt
\\
&\leq&C\sum_{j}2^{-\alpha j}2^{p\beta j}|u_{0,j}|_{p}^{p}.
\end{eqnarray*}

By H\"{o}lder's inequality,%
\begin{equation*}
B(t)\leq \bigg\{\sum_{j\in \mathbb{N}_{0}\setminus J}e^{-c2^{\alpha j}t}%
\bigg\}^{\frac{1}{q}}\bigg\{\sum_{j\in \mathbb{N}_{0}\setminus
J}e^{-c2^{\alpha j}t}2^{\beta pj}|u_{0,j}|_{p}^{p}\bigg\}^{\frac{2}{p}}
\end{equation*}%
with $q=\frac{p}{p-2}$. Since $e^{-c2^{\alpha j}t}$ is decreasing in $j$, we
have 
\begin{equation*}
\sum_{j\in \mathbb{N}_{0}\setminus J}e^{-c2^{\alpha j}t}\leq \int_{t\geq
2^{-\alpha r}}e^{-c2^{-\alpha }2^{\alpha r}t}dr\leq C.
\end{equation*}%
Therefore, 
\begin{eqnarray*}
\int_{0}^{T}B(t)^{\frac{p}{2}}dt &\leq &C\sum_{j}2^{\beta
pj}|u_{0,j}|_{p}^{p}\int_{0}^{T}e^{-c2^{\alpha j}t}dt \\
&\leq &C\sum_{j}2^{-\alpha j}2^{\beta pj}|u_{0,j}|_{p}^{p}.
\end{eqnarray*}%
Finally,%
\begin{eqnarray*}
|T_{\cdot }u_{0}|_{H_{p}^{\beta }(E)}^{p} &\leq &C\bigg(\int_{0}^{T}A(t)^{%
\frac{p}{2}}dt+\int_{0}^{T}B(t)^{\frac{p}{2}}dt\bigg) \\
&\leq &C\sum_{j}2^{-\alpha j}2^{\beta
pj}|u_{0,j}|_{p}^{p}=C|u_{0}|_{B_{pp}^{\beta -\frac{\alpha }{p}}(\mathbf{R}%
^{d})}^{p}.
\end{eqnarray*}%
The lemma is proved.
\end{proof}

For a bounded measurable $m(y),y\in \mathbf{R}^{d},$ and $\alpha \in (0,2)$,
set for $v\in \mathcal{S}(\mathbf{R}^{d}),x\in \mathbf{R}^{d},$%
\begin{equation*}
\mathcal{L}v(x)=\mathcal{L}_{\alpha }v(x)=\int \nabla ^{\alpha }v(x)m(y)%
\frac{dy}{|y|^{d+\alpha }}.
\end{equation*}%
We will need the following continuity estimate (see \cite{babo0} for a
symmetric case, Theorem 2.1 in \cite{kimdong} for a general case using H\"{o}%
lder estimates, and Lemma 10 in \cite{mikprag2} for a direct proof).

\begin{lemma}
\label{lem2}(Lemma 10, \cite{mikprag2}) Let $|m(y)|\leq K,y\in \mathbf{R}%
^{d},p>1,$ and $\alpha \in (0,2).$ Assume%
\begin{equation*}
\int_{r\leq |y|\leq R}ym(y)\frac{dy}{|y|^{d+\alpha }}=0
\end{equation*}%
for any $0<r<R$ if $\alpha =1$. Then there is a constant $C$ such that%
\begin{equation*}
|\mathcal{L}_{\alpha }u|_{p}\leq CK|\partial ^{\alpha }u|_{p},u\in L_{p}(%
\mathbf{R}^{d}).
\end{equation*}
\end{lemma}

\subsection{Solution for smooth input functions}

\begin{theorem}
\label{le3}Let $\alpha\in(0,2],\ p\geq 2$ and Assumption \emph{A}$_0$ be
satisfied. Let $u_{0}\in \mathfrak{D}_{p}(\mathbf{R}^{d})$ be $\mathcal{F}%
_{0}$-measurable, $f\in \mathfrak{D}_{p}(E)$, $\Phi \in \mathfrak{D}%
_{2,p}(E)\cap \mathfrak{D}_{p,p}(E)$ and $h\in \mathfrak{D}_{p}(E,Y)$.

Then there is a unique strong solution $u\in \mathfrak{D}_{p}(E)$ of \emph{(%
\ref{eq15})}. Moreover, $\mathbf{P}$-a.s. $u(t,x)$ is cadlag in $t$, smooth
in $x$ and the following assertions hold:

\emph{(i)} for each multiindex $\gamma \in \mathbf{N}_{0}^{d}$ and $(t,x)\in
E$ $\mathbf{P}$-a.s. 
\begin{eqnarray*}
\partial _{x}^{\gamma }u(t,x) &=&T_{t}^{\lambda }\partial ^{\gamma
}u_{0}(x)+R_{\lambda }\partial _{x}^{\gamma }f(t,x)+\widetilde{R}_{\lambda
}\partial _{x}^{\gamma }\Phi (t,x) \\
&&+\bar{R}_{\lambda }\partial _{x}^{\gamma }h(t,x);
\end{eqnarray*}

\emph{(ii)} for each multiindex $\gamma \in \mathbf{N}_{0}^{d}$ 
\begin{eqnarray*}
|\partial ^{\gamma }u|_{\mathbb{L}_{p}(E)} &\leq &C\biggl\{\rho _{\lambda
}^{1/p}|\partial ^{\gamma }u_{0}|_{\mathbb{L}_{p}(\mathbf{R}^{d})}+\rho
_{\lambda }|\partial ^{\gamma }f|_{\mathbb{L}_{p}(E)} \\
&&+\sum_{r=2,p}\rho _{\lambda }^{1/r}|\partial ^{\gamma }\Phi |_{\mathbb{L}%
_{r,p}(E)}+\rho _{\lambda }^{1/2}|\partial ^{\gamma }h|_{\mathbb{L}%
_{p}(E,Y)}^{p}\biggr\},
\end{eqnarray*}%
where $\rho _{\lambda }=T\wedge \frac{1}{\lambda }$ and the constant $%
C=C(\alpha ,p,d,|\gamma |,K,\delta )$;

\emph{(iii)} for each $\beta \in \mathbf{R}$ 
\begin{eqnarray*}
|u|_{\mathbb{H}_{p}^{\beta +\alpha }(E)} &\leq &C\Bigl\{|u_{0}|_{\mathbb{B}%
_{pp}^{\beta +\alpha -\frac{\alpha }{p}}(\mathbf{R}^{d})}+|f|_{\mathbb{H}%
_{p}^{\beta }(E)}+|\Phi |_{\mathbb{H}_{2,p}^{\beta +\frac{\alpha }{2}}(E)} \\
&&+|\Phi |_{\mathbb{B}_{p,pp}^{\beta +\alpha -\frac{\alpha }{p}}(E)}+|h|_{%
\mathbb{H}_{p}^{\beta +\alpha /2}(E,Y)}\Bigr\},
\end{eqnarray*}%
where the constant $C=C(\alpha ,\beta ,p,d,T,K,\delta )$.
\end{theorem}

\begin{proof}
We follow the arguments in \cite{MiPtoap}, \cite{MiP09}. Denote by $%
C_{p}^{\infty }(E)$ the set of all $\mathcal{R}(\mathbb{F})\otimes \mathcal{B%
}(\mathbf{R}^{d})$-measurable random functions $v(t,x)$ on $E$ such that $%
\mathbf{P}$-a.s. for all $t\in \lbrack 0,T]$ $u(t,x)$ is infinitely
differentiable in $x$ and for every multiindex $\gamma \in \mathbf{N}%
_{0}^{d} $ 
\begin{equation*}
\sup_{(t,x)\in E}\mathbf{E}|D_{x}^{\gamma }v(t,x)|^{p}<\infty .
\end{equation*}%
According to the definition of $\mathfrak{D}_{p}(E)$, we have $\mathfrak{D}%
_{p}(E)\subset C_{p}^{\infty }(E)$.

Let $u_{0}=0$. Since for every multiindex $\gamma \in \mathbf{N}_{0}^{d}$ 
\begin{eqnarray*}
\sup_{(t,x)\in E}\mathbf{E}\biggl\{|D_{x}^{\gamma }f(t,x)|^{p}
&+&\sum_{r=2,p}\biggl(\int_{U}\bigl\vert D_{x}^{\gamma }\Phi (t,x,\upsilon )%
\bigr\vert^{r}\Pi (d\upsilon )\biggr)^{\frac{p}{r}} \\
&+&\bigl\vert D_{x}^{\gamma }h(t,x)\bigr\vert_{Y}^{p}\biggr\}<\infty ,
\end{eqnarray*}%
by Lemma 8 in \cite{MiPtoap} and Lemma 7 in \cite{MiP09} there is a unique $%
u\in C_{p}^{\infty }(E)$ solving (\ref{eq15}), $u(t,x)$ is cadlag in $t$ and
the assertion (i) holds with $\gamma =0$. Moreover (see equation~(20) in the
proof of Lemma 8 in \cite{MiPtoap} and the proof of Lemma 7 in \cite{MiP09}%
), for every $\gamma \in \mathbf{N}_{0}^{d}$ and $(t,x)\in E$ we have $%
\mathbf{P}$-a.s.%
\begin{eqnarray*}
D_{x}^{\gamma }u(t,x) &=&\int_{0}^{t}\bigl[A_{0}^{(\alpha )}D_{x}^{\gamma
}u-\lambda D_{x}^{\gamma }u+D_{x}^{\gamma }f\bigr](s,x)ds \\
&&+\int_{0}^{t}\int_{U}D_{x}^{\gamma }\Phi (s,x,\upsilon )\eta (ds,d\upsilon
)+\int_{0}^{t}D_{x}^{\gamma }h(s,x)dW_{s}\,.
\end{eqnarray*}

Applying Lemma 8 in \cite{MiPtoap} and Lemma 7 in \cite{MiP09} again, we get
the assertion (i) for arbitrary $\gamma \in \mathbf{N}_{0}^{d}$. The
estimates (ii) and (iii) follow by Lemmas \ref{le1} and \ref{le2}. The
assertion (iii) and embedding theorem imply that $u\in \mathfrak{D}_{p}(E)$.
Using Lemma 3.2 in \cite{MiP922}, we get that there is a constant $C$ such
that for every $\upsilon \in \mathbb{H}_{p}^{\beta +\alpha }(E)$ 
\begin{equation*}
|A_{0}^{(\alpha )}\upsilon |_{\mathbb{H}_{p}^{\beta }(E)}\leq C|\upsilon |_{%
\mathbb{H}_{p}^{\beta +\alpha }(E)}.
\end{equation*}%
Hence, $u\in \mathfrak{D}_{p}(E)$ is a unique strong solution of (\ref{eq15}%
).

The case $u_{0}\neq 0$ is considered as above repeating the proof of Lemma 8
in \cite{MiPtoap} with obvious changes. The theorem is proved.
\end{proof}

%%%%%%%%%%%%%%%%%%%%%%%%%%%%%%%%%%%%%%%%%%%%%%%%%%%%%%%%%%%%%%%%%%%%%%%%%%%%%%%%%%%%%%%%%%%%%%%%%%%%%%%%%%%%%%%%%%%%%%%%%%%%%%%%%%
%%%%%%%%%%%%%%%%%%%%%%%%%%%%%%%%%%%%%%%%%%%%%%%%%%%%%%%%%%%%%%%%%%%%%%%%%%%%%%%%%%%%%%%%%%%%%%%%%%%%%%%%%%%%%%%%%%%%%%%%%%%%%%%%%%

\section{Model problem. Partial case II\label{sei}}

In this section, we consider the following partial case of equation (\ref%
{intr1}):%
\begin{eqnarray}
du(t,x) &=&\big(A^{(\alpha )}u-\lambda u+f\big)(t,x)dt+\int_{U}\Phi
(t,x,\upsilon )\eta (dt,d\upsilon )  \label{pr3} \\
&&+\int_{\mathbf{R}_{0}^{d}}g(t,x,y)q^{(\alpha )}(dt,dy)1_{\alpha \in
(0,2)}+h(t,x)dW_{t},  \notag \\
u(0,x) &=&u_{0}(x).  \notag
\end{eqnarray}%
We prove Theorem~\ref{main1} which is a partial case of the following
statement.

\begin{theorem}
\label{thm15}Let $\alpha \in (0,2],\ \beta \in \mathbf{R},\ p\geq 2$ and
Assumptions~\emph{A(i)-(ii)} be satisfied with $\sigma ^{i}=0,i=1,\ldots ,d$%
. Assume $\mathbf{P}$-a.s. for all $t\in \lbrack 0,T]$ and $y\in \mathbf{R}%
_{0}^{d}$ 
\begin{eqnarray*}
&&m^{(\alpha )}(t,y)\geq m_{0}^{(\alpha )}(t,y)\mbox{\quad if }\alpha \in
(0,2), \\
&&\big(B^{ij}(t)\big)y_{i}y_{j}\geq \delta |y|^{2}\mbox{\quad if }\alpha =2,
\end{eqnarray*}%
where the functions $m_{0}^{(\alpha )}$ satisfy Assumption A$_{0}$. Let $%
u_{0}\in \mathbb{B}_{pp}^{\beta +\alpha -\frac{\alpha }{p}}(\mathbf{R}^{d})$
be $\mathcal{F}_{0}$-measurable, $f\in \mathbb{H}_{p}^{\beta }(E),\ \Phi \in 
\mathbb{B}_{p,pp}^{\beta +\alpha -\frac{\alpha }{p}}(E)\cap \mathbb{H}%
_{2,p}^{\beta +\frac{\alpha }{2}}(E)$,\ $g\in \bar{\mathbb{B}}_{p,pp}^{\beta
+\alpha -\frac{\alpha }{p}}(E)\cap \bar{\mathbb{H}}_{2,p}^{\beta +\frac{%
\alpha }{2}}(E)$ and $h\in \mathbb{H}_{p}^{\beta +\alpha /2}(E,Y)$.

Then there is a unique strong solution $u\in \mathbb{H}_{p}^{\beta +\alpha
}(E)$ of \emph{(\ref{pr3})}. Moreover, there is a constant $C=C(\alpha
,\beta ,p,d,T,K,\delta )$ such that 
\begin{eqnarray}
|u|_{\mathbb{H}_{p}^{\beta +\alpha }(E)} &\leq &C\Bigl(|u_{0}|_{\mathbb{B}%
_{pp}^{\beta +\alpha -\frac{\alpha }{p}}(E)}+|f|_{\mathbb{H}_{p}^{\beta
}(E)}+|\Phi |_{{\mathbb{H}}_{2,p}^{\beta +\frac{\alpha }{2}}(E)}+|\Phi |_{{%
\mathbb{B}}_{p,pp}^{\beta +\alpha -\frac{\alpha }{p}}(E)}  \label{eq20} \\
&&+\big(|g|_{\bar{\mathbb{H}}_{2,p}^{\beta +\frac{\alpha }{2}}(E)}+|g|_{\bar{%
\mathbb{B}}_{p,pp}^{\beta +\alpha -\frac{\alpha }{p}}(E)}\big)+|h|_{\mathbb{H%
}_{p}^{\beta +\alpha /2}(E,Y)}\Bigr).  \notag
\end{eqnarray}
\end{theorem}

First, we consider (\ref{pr3}) for smooth in $x$ input functions $%
u_0,f,\Phi,g$ and $h$.

\begin{lemma}
\label{le16} Let $\alpha\in(0,2],\ \beta \in \mathbf{R},\ p\geq 2$ and
Assumption~\emph{A} be satisfied with $\sigma^{i}=0,\ i=1,\ldots ,d,$ and $%
l^{(\alpha )}=0$ in \emph{A(iii)}. Let $u_{0}\in \mathfrak{D}_{p}(\mathbf{R}%
^{d})$ be $\mathcal{F}_{0}$-measurable, $f\in \mathfrak{D}_{p}(E),\ \Phi \in 
\mathfrak{D}_{2,p}(E)\cap \mathfrak{D}_{p,p}(E)$,\ $g \in \bar{\mathfrak{D}}%
_{2,p}(E)\cap \bar{\mathfrak{D}}_{p,p}(E)$ and $h\in \mathfrak{D}_{p}(E,Y)$.

Then there is a unique strong solution $u\in \mathfrak{D}_{p}(E)$ of \emph{(%
\ref{pr3})}. Moreover, $\mathbf{P}$-a.s. $u(t,x)$ is cadlag in $t$, smooth
in $x$ and the following assertions hold:

\emph{(i)} for every multiindex $\gamma \in \mathbf{N}_{0}^{d}$ 
\begin{eqnarray}
|D^{\gamma }u|_{\mathbb{L}_{p}(E)} &\leq &C\biggl\{\rho _{\lambda
}^{1/p}|D^{\gamma }u_{0}|_{\mathbb{L}_{p}(\mathbf{R}^{d})}+\rho _{\lambda
}|D^{\gamma }f|_{\mathbb{L}_{p}(E)}  \label{eq21} \\
&&+\sum_{r=2,p}\rho _{\lambda }^{1/r}\bigl(|D^{\gamma }\Phi |_{\mathbb{L}%
_{r,p}(E)}+|D^{\gamma }g|_{\mathbb{\bar{L}}_{r,p}(E)}\bigr)  \notag \\
&&+\rho _{\lambda }^{1/2}|D^{\gamma }h|_{\mathbb{L}_{p}(E,Y)}\biggr\}. 
\notag
\end{eqnarray}%
where $\rho _{\lambda }=T\wedge \frac{1}{\lambda }$ and the constant $%
C=C(\alpha ,p,d,|\gamma |,K,\delta )$;

\emph{(ii)} the estimate \emph{(\ref{eq20})} holds for every $\beta\in%
\mathbf{R}$.
\end{lemma}

\begin{proof}
\emph{Existence.} 1$^{0}$. First, we consider the equation (\ref{pr3}) with $%
A^{(\alpha )}$ replaced by $A_{0}^{(\alpha )}$ (equivalently, $m^{(\alpha )}$
in the definition of $A^{(\alpha )}$ is replaced by $m_{0}^{(\alpha )}$) .

By Lemma 14.50 and Theorem 14.56 in \cite{Jac79}, there is a $\mathbf{R}^{d}$%
-valued $\mathcal{R}(\mathbb{F})\otimes \mathcal{B}(\mathbf{R}_{0})$%
-measurable random function $c^{(\alpha )}(t,z)$ on $[0,T]\times \mathbf{R}%
_{0}$ satisfying (\ref{eq12}) and a Poisson point measure $\tilde{p}(dt,dz)$
on $([0,\infty )\times \mathbf{R}_{0},\mathcal{B}([0,\infty ))\otimes 
\mathcal{B}(\mathbf{R}_{0}))$, possibly on an extended probability space$,$
such that 
\begin{equation*}
p^{(\alpha )}(dt,dy)=\int_{\mathbf{R}_{0}}1_{dy}(c^{(\alpha )}(t,z))\tilde{p}%
(dt,dz)
\end{equation*}%
and 
\begin{equation*}
q^{(\alpha )}(dt,dy)=\int_{\mathbf{R}_{0}}1_{dy}(c^{(\alpha )}(t,z))\tilde{q}%
(dt,dz),
\end{equation*}%
where $\tilde{q}(dt,dz)=\tilde{p}(dt,dz)-\frac{dzdt}{z^{2}}$. Hence, for
every $g\in \bar{\mathfrak{D}}_{2,p}(E)\cap \bar{\mathfrak{D}}_{p,p}(E)$, we
have 
\begin{equation*}
\int_{0}^{t}\int_{\mathbf{R}_{0}^{d}}g(s,x,y)q^{(\alpha
)}(ds,dy)=\int_{0}^{t}\int_{\mathbf{R}_{0}}\tilde{g}(s,x,z)\tilde{q}(ds,dz),
\end{equation*}%
where $\tilde{g}(s,x,z)=g(s,x,c^{(\alpha )}(s,z))$. Since the point measures 
$\tilde{p}$ and $\eta $ have no common jumps, the problem (\ref{pr3})
reduces to the case of a single point measure on $[0,\infty )\times V$,
where $V$ is the sum of $U$ and $\mathbf{R}_{0}.$ Therefore,Theorem \ref{le3}
applies and all the assertions of the Lemma follow in the case $A^{(\alpha
)}=A_{0}^{(\alpha )}$.

2$^{0}$. Let $(\overline{\Omega },\overline{\mathcal{F}},\overline{\mathbf{P}%
})$ be a complete probability space with a filtration of $\sigma $-algebras $%
\overline{\mathbb{F}}=(\overline{\mathcal{F}}_{t})$ satisfying the usual
conditions. Let $\bar{p}(dt,dz)$ be an $\overline{\mathbb{F}}$-adapted
Poisson measure on $([0,\infty )\times \mathbf{R}_{0},\mathcal{B}([0,\infty
))\otimes \mathcal{B}(\mathbf{R}_{0}))$ with the compensator $dtdz/z^{2}$
and $\overline{W}_{t}$ be an independent standard $\overline{\mathbb{F}}$%
-adapted Wiener process in $\mathbf{R}^{d}$.

We introduce the product of probability spaces 
\begin{equation*}
(\widetilde{\Omega},\widetilde{\mathcal{F}},\widetilde{\mathbf{P}}) =
(\Omega\times\overline{\Omega},\mathcal{F}\otimes\overline{\mathcal{F}},%
\mathbf{P}\times\overline{\mathbf{P}}).
\end{equation*}%
Let $\widetilde{\mathcal{F}}^{\prime }$ be the completion of $\widetilde{%
\mathcal{F}}$. Let $\widetilde{\mathbb{F}}^{\prime }=(\widetilde{\mathcal{F}}%
^{\prime }_t)$, $\widetilde{\mathbb{F}}^{\prime \prime }=(\widetilde{%
\mathcal{F}}^{\prime \prime }_t)$ and $\widetilde{\mathbb{F}}^{\prime \prime
\prime }=(\widetilde{\mathcal{F}}^{\prime \prime \prime }_t)$ be the usual
augmentations of $(\mathcal{F}_t\otimes\overline{\mathcal{F}}_t)$, $(%
\mathcal{F}\otimes\overline{\mathcal{F}}_t)$ and $(\mathcal{F}_t\otimes%
\overline{\mathcal{F}})$, respectively (see \cite{delmey}).

Obviously,%
\begin{equation*}
\bar{q}(dt,dz)=\bar{p}(dt,dz)-\frac{dzdt}{z^{2}}
\end{equation*}%
is an $(\mathbb{\widetilde{F}}^{\prime },\mathbf{\widetilde{P}})$- and $(%
\mathbb{\widetilde{F}}^{\prime \prime },\mathbf{\widetilde{P}})$-martingale
measure. Also, $q^{(\alpha )}(dt,dy)$ and $\eta (dt,d\upsilon )$ are $(%
\mathbb{\widetilde{F}}^{\prime },\mathbf{\widetilde{P}})$- and $(\mathbb{%
\widetilde{F}}^{\prime \prime \prime },\mathbf{\widetilde{P}}$)-martingale
measures.

By Lemma 14.50 in \cite{Jac79}, there is a $\mathcal{R}(\mathbb{F})\otimes 
\mathcal{B}(\mathbf{R}_{0})$-measurable $\mathbf{R}^{d}$-valued function $%
c_{0}^{(\alpha )}(t,z)$ such that%
\begin{equation*}
\lbrack m^{(\alpha )}(t,y)-m_{0}^{(\alpha )}(t,y)]\frac{dy}{|y|^{d+\alpha }}%
=\int_{\mathbf{R}_{0}}1_{dy}(c_{0}^{(\alpha )}(t,z))\frac{dz}{z^{2}}
\end{equation*}%
with $\alpha \in (0,2).$

Let $\sigma _{\delta }(t)$ be a symmetric square root of the matrix $%
B(t)-\delta I$. We introduce the $\widetilde{\mathbb{F}}^{\prime }$-adapted
processes $Y_{t}^{(\alpha )},\ t\in \lbrack 0,T]$, defined by 
\begin{eqnarray*}
Y_{t}^{(\alpha )} &=&\int_{0}^{t}\int_{\mathbf{R}_{0}}\chi ^{(\alpha
)}(c_{0}^{(\alpha )}(s,z))c_{0}^{(\alpha )}(s,z)\bar{q}(ds,dz) \\
&&+\int_{0}^{t}\int_{\mathbf{R}_{0}}[1-\chi ^{(\alpha
)}(c_{0}^{(a)}(s,z))]c_{0}^{(\alpha )}(s,z)\bar{p}(ds,dz)
\end{eqnarray*}%
for $\alpha \in (0,2)$ and 
\begin{equation*}
Y_{t}^{(2)}=\int_{0}^{t}\sigma _{\delta }(s)d\overline{W}_{s}.
\end{equation*}

Let us consider the problem%
%\begin{equation}
%\left\{%
\begin{eqnarray}
dw(t,x) &=&\big[A_{0}^{(\alpha )}w(t,x)-\lambda w(t,x)+f\big(%
t,x-Y_{t}^{(\alpha )}\big)\big]dt  \notag \\
&&+\int_{U}\Phi \big(t,x-Y_{t-}^{(\alpha )},\upsilon \big)\eta (dt,d\upsilon
)  \notag \\
&&+\int_{\mathbf{R}_{0}^{d}}g\big(t,x-Y_{t-}^{(\alpha )},y\big)q^{(\alpha
)}(dt,dy)1_{\alpha \in (0,2)}  \label{eq22} \\
&&+h\big(t,x-Y_{t}^{(\alpha )}\big)dW_{t}\,,  \notag \\
w(0,x) &=&u_{0}(x).  \notag
\end{eqnarray}%
%
%
%
%
%
%
%
%
%
%
%
%
%
%
%
%
%
%
%
%
%
%
%
%
%
%
%
%
%
%
%
%
%\right.
%\label{eq22}
%\end{equation}%
Obviously, $f(t,x-Y_{t}^{(\alpha )})\in \mathfrak{D}_{p}(E),\Phi
(t,x-Y_{t}^{(\alpha )},\upsilon )\in \mathfrak{D}_{2,p}(E)\cap \mathfrak{D}%
_{p,p}(E),g(t,x-Y_{t}^{(\alpha )},y)\in \mathfrak{\bar{D}}_{2,p}(E)\cap 
\mathfrak{\bar{D}}_{p,p}(E)$ and $h(t,x-Y_{t}^{(\alpha )})\in \mathfrak{D}%
_{p}(E,Y)$, where the classes $\mathfrak{D}_{p}(E),\ \mathfrak{D}_{r,p}(E)$, 
$\overline{\mathfrak{D}}_{r,p}(E)$ and $\mathfrak{D}_{p}(E,Y)$ are defined
on the extended probability space $(\widetilde{\Omega },\widetilde{\mathcal{F%
}},\widetilde{\mathbf{P}})$ with the filtration $\widetilde{\mathbb{F}}%
^{\prime }$. According to the first part of the proof, there is a unique
strong solution $w\in \mathfrak{D}_{p}(E)$ of (\ref{eq22}). Moreover, $%
w(t,x) $ is cadlag in $t$, smooth in $x$ and possesses the properties (i)
and (ii) with all the norms defined on the extended probability space. Since
the norms entering the estimates (i) and (ii) are invariant with respect to
random shifts of the space variable $x\in \mathbf{R}^{d}$, we conclude that
the norms $|\partial ^{\gamma }w|_{\mathbb{L}_{p}(E)}$ and $|w|_{\mathbb{H}%
_{p}^{\beta +\alpha }(E)}$ defined on the extended probability space do not
exceed the right-hand sides of the estimates (i) and (ii) defined on the
original probability space.

Applying the Ito-Wentzel formula (see Proposition 1 of \cite{mik1} and note
that $Y_{s}^{(\alpha )}$ and $w(s,x)$ have no common jumps), we have%
\begin{eqnarray*}
w(t,x+Y_{t}^{(\alpha )}) &=&u_{0}(x)+\int_{0}^{t}\Bigl[A_{0}^{(\alpha )}w%
\big(s,x+Y_{s}^{(\alpha )}\big)-\lambda w\big(s,x+Y_{s}^{(\alpha )}\big) \\
&&+f(s,x)+\frac{1}{2}\big(B(s)-\delta I\big)^{ij}w_{x_{i}x_{j}}(s,x+Y_{s}^{(%
\alpha )})1_{\alpha =2}\Bigr ]ds \\
&&+\int_{0}^{t}\nabla w\big(s-,x+Y_{s-}^{(\alpha )}\big)dY_{s}^{(\alpha )} \\
&&+\sum_{s\leq t}\Bigl[w\big(s-,x+Y_{s}^{(\alpha )}\big)-w\big(%
s-,x+Y_{s-}^{(\alpha )}\big) \\
&&-\big(\nabla w(s-,x+Y_{s-}^{(\alpha )}),Y_{s}^{(\alpha )}-Y_{s-}^{(\alpha
)}\big)\Bigr ]1_{\alpha \in (0,2)} \\
&&+\int_{0}^{t}\int_{U}\Phi (s,x,\upsilon )\eta (ds,d\upsilon ) \\
&&+\int_{0}^{t}\int_{\mathbf{R}_{0}^{d}}g(s,x,y)q^{(\alpha
)}(ds,dy)1_{\alpha \in (0,2)}+\int_{0}^{t}h(s,x)dW_{s}\,.
\end{eqnarray*}%
Thus 
\begin{eqnarray}
&&w(t,x+Y_{t}^{(\alpha )})  \label{fo1} \\
&=&u_{0}(x)+\int_{0}^{t}\Bigl[A^{(\alpha )}w\big(s,x+Y_{s}^{(\alpha )}\big)%
-\lambda w\big(s,x+Y_{s}^{(\alpha )}\big)+f(s,x)\Bigr]ds  \notag \\
&&+\int_{0}^{t}\int_{\mathbf{R}_{0}}\Bigl[w\big(s-,x+Y_{s-}^{(\alpha
)}+c_{0}^{(\alpha )}(s,z)\big)-w\big(s-,x+Y_{s-}^{(\alpha )}\big)\Bigr]\bar{q%
}(ds,dz)1_{\alpha \in (0,2)}  \notag \\
&&\quad +\int_{0}^{t}\nabla w(s,x+Y_{s}^{(\alpha )})\sigma _{\delta }(s)d%
\overline{W}_{s}\,1_{\alpha =2}+\int_{0}^{t}\int_{U}\Phi (s,x,\upsilon )\eta
(ds,d\upsilon )  \notag \\
&&\quad +\int_{0}^{t}\int_{\mathbf{R}_{0}^{d}}g(s,x,y)q^{(\alpha
)}(ds,dy)1_{\alpha \in (0,2)}+\int_{0}^{t}h(s,x)dW_{s}.  \notag
\end{eqnarray}

Let 
\begin{equation*}
\widetilde{w}(t,x)=w\big(t,x+Y_{t}^{(\alpha )}\big),\quad u(t,x)=\overline{%
\mathbf{E}}\widetilde{w}(t,x),
\end{equation*}%
where for a measurable integrable function $F$ on $\tilde{\Omega}=\Omega
\times \bar{\Omega}$ we denote%
\begin{equation*}
\overline{\mathbf{E}}F=\int F(\omega ,\bar{\omega})\mathbf{\bar{P}}(d\bar{%
\omega}).
\end{equation*}%
Obviously, $u\in $ $\mathfrak{D}_{p}(E)$, and by H\"{o}lders inequality, 
\begin{equation*}
|\partial ^{\gamma }u|_{\mathbb{L}_{p}(E)}\leq |\partial ^{\gamma }%
\widetilde{w}|_{\mathbb{L}_{p}(E)},\quad |u|_{\mathbb{H}_{p}^{\beta +\alpha
}(E)}\leq |\widetilde{w}|_{\mathbb{H}_{p}^{\beta +\alpha }(E)},
\end{equation*}%
where the norms $|\partial ^{\gamma }\widetilde{w}|_{\mathbb{L}_{p}(E)}$ and 
$|\partial ^{\gamma }\widetilde{w}|_{\mathbb{H}_{p}^{\beta +\alpha }(E)}$
defined on the extended probability space coincide with the norms $|\partial
^{\gamma }{w}|_{\mathbb{L}_{p}(E)}$ and $|{w}|_{\mathbb{H}_{p}^{\beta
+\alpha }(E)}$ and do not exceed the right-hand sides of the estimates (i)
and (ii). Therefore, the function $u$ satisfies the estimates (i) and (ii).
Moreover, $u$ is cadlag in $t$, and smooth in $x$. Since a $\mathcal{R}(%
\mathbb{\widetilde{F}}^{\prime })$-measurable process is $\mathcal{R}(%
\mathbb{\widetilde{F}}^{\prime \prime })$- and $\mathcal{R}(\mathbb{%
\widetilde{F}}^{\prime \prime \prime })$-measurable as well, taking
expectation $\mathbf{\bar{E}}$ of both sides of (\ref{fo1}) and applying
Lemma \ref{lep1}, we see that $u$ satisfies (\ref{pr3}).

\medskip \emph{Uniqueness.} Let $u_{i}\in \mathfrak{D}_{p}(E)$, $i=1,2$, be
two strong solutions of (\ref{pr3}). Then $u=u_{1}-u_{2}$ is a strong
solution to the problem 
\begin{eqnarray}
du(t,x) &=&\big(A^{(\alpha )}u-\lambda u\big)(t,x)dt\mbox{\ \ in\ }E,  \notag
\\
u(0,x) &=&0\mbox{\ \ in\ }\mathbf{R}^{d}.  \label{eq230}
\end{eqnarray}%
Considering (\ref{eq230}) separately for every $\omega \in \Omega $, without
loss of generality we can assume that the coefficients $m^{(\alpha )},B,b$
of the operator $A^{(\alpha )}$ and the function $u$ are non-random.

We fix arbitrary $(t_0,x)\in E$ and introduce the processes $%
Z_t^{(\alpha)},\ t\in[0,t_0],\ \alpha\in(0,2]$, defined on some probability
space by 
\begin{eqnarray*}
Z_t^{(\alpha)}&=&\int_0^t\int_{\mathbf{R}_0^d}\chi^{(\alpha)}(y)yq_{%
\alpha}(ds,dy)+ \int_0^t\int_{\mathbf{R}_0^d}\big[1-\chi^{(\alpha)}(y)\big]%
yp_{\alpha}(ds,dy) \\
&&+\int_0^t b(s)ds\,1_{\alpha=1}
\end{eqnarray*}%
for $\alpha\in(0,2)$ and 
\begin{equation*}
Z_t^{(2)}=\int_0^t \hat{\sigma}(s)d\widehat{W}_s .
\end{equation*}
Here, $p_{\alpha}(dt,dy)$ is a Poisson point measure on $([0,t_0]\times%
\mathbf{R}_0^d,\mathcal{B}([0,t_0])\otimes\mathcal{B}(\mathbf{R}_0^d))$ with
the compensator $m^{(\alpha)}(t_0-t,y)dydt/|y|^{d+\alpha}$, 
\begin{equation*}
q_{\alpha}(dt,dy)=p_{\alpha}(dt,dy)-m^{(\alpha)}(t_0-t,y)\frac{dydt}{%
|y|^{d+\alpha}}
\end{equation*}
is a martingale measure, $\hat{\sigma}(t)$ is a symmetric square root of the
matrix $B(t)$ and $\widehat{W}_t$ is a standard Wiener process in $\mathbf{R}%
^d$. By Ito's formula, 
\begin{eqnarray*}
-u(t_0,x)&=& e^{-\lambda t_0} u\big(0,x+Z_{t_0}^{(\alpha)}\big)-u(t_0,x) \\
&=&\int_0^{t_0}e^{-\lambda t}\Bigl( -\frac{\partial u}{\partial t}%
+A^{(\alpha)}u-\lambda u \Bigr)\big( t_0-t,x+Z_t^{(\alpha)}\big)dt=0.
\end{eqnarray*}%
Since $(t_0,x)\in E$ was arbitrary, $u=0$ on $E$.

The lemma is proved.
\end{proof}

\begin{proof}[Proof of Theorem \protect\ref{thm15}]
\emph{Existence}. According to Lemma~\ref{lemd}, there is a sequence of
input functions $(u_{0n},f_{n},\Phi _{n},g_{n},h_{n})$, $n=1,2,\ldots ,$
such that $u_{0n}\in \mathfrak{D}_{p}(\mathbf{R}^{d})$, $f_{n}\in \mathfrak{D%
}_{p}(E)$, $\Phi _{n}\in \mathfrak{D}_{2,p}(E)\cap \mathfrak{D}_{p,p}(E)$, $%
g_{n}\in \overline{\mathfrak{D}}_{2,p}(E)\cap \overline{\mathfrak{D}}%
_{p,p}(E)$, $h_{n}\in \mathfrak{D}_{p}(E,Y)$ and 
\begin{eqnarray}
&&|u_{0}-u_{0n}|_{\mathbb{B}_{pp}^{\beta +\alpha -\frac{\alpha }{p}}(\mathbf{%
R}^{d})}+|f-f_{n}|_{\mathbb{H}_{p}^{\beta }(E)}+|\Phi -\Phi _{n}|_{\mathbb{H}%
_{2,p}^{\beta +\frac{\alpha }{2}}(E)}+|\Phi -\Phi _{n}|_{\mathbb{B}%
_{p,pp}^{\beta +\alpha -\frac{\alpha }{p}}(E)}  \notag \\
&&\quad +|g-g_{n}|_{\overline{\mathbb{H}}_{2,p}^{\beta +\frac{\alpha }{2}%
}(E)}+|g-g_{n}|_{\overline{\mathbb{B}}_{p,pp}^{\beta +\alpha -\frac{\alpha }{%
p}}(E)}+|h-h_{n}|_{{\mathbb{H}}_{p}^{\beta +\alpha /2}(E,Y)}\rightarrow 0
\label{eq24}
\end{eqnarray}%
as $n\rightarrow \infty $. By Lemma \ref{le16}, for every $n$ there is a
strong solution $u_{n}\in \mathfrak{D}_{p}(E)$ of (\ref{pr3}) with the input
functions $u_{0n},f_{n},\Phi _{n},g_{n},h_{n}$. Since (\ref{pr3}) is a
linear equation, using the estimate (ii) of Lemma~\ref{le16} we derive that $%
(u_{n})$ is a Cauchy sequence in $\mathbb{H}_{p}^{\beta +\alpha }(E)$.
Hence, there is a function $u\in \mathbb{H}_{p}^{\beta +\alpha }(E)$ such
that $|u_{n}-u|_{\mathbb{H}_{p}^{\beta +\alpha }(E)}\rightarrow 0$ as $%
n\rightarrow \infty $.

Passing to the limit in (\ref{eq20}) with $u,u_0,f,\Phi,g,h$ replaced by $%
u_n,u_{0n},f_n,\Phi_n,g_n,h_n$ and using (\ref{eq23}), we get the estimate (%
\ref{eq20}).

Passing to the limit in the equality (see Definition \ref{def1}) 
\begin{eqnarray*}
\big\langle J^{\beta }u_{n}(t,\cdot ),\varphi \big\rangle &=&\big\langle %
J^{\beta }u_{0},\varphi \big\rangle+\int_{0}^{t}\Bigl[\big\langle(A^{(\alpha
)}-\lambda )J^{\beta }u_{n}(s,\cdot ),\varphi \big\rangle+\big\langle %
J^{\beta }f(s,\cdot ),\varphi \big\rangle\Bigr]ds \\
&&+\int_{0}^{t}\int_{U}\big\langle J^{\beta }\Phi _{n}(s,\cdot ,\upsilon
),\varphi \big\rangle\eta (ds,d\upsilon ) \\
&&+\int_{0}^{t}\int_{\mathbf{R}_{0}^{d}}\big\langle J^{\beta }g_{n}(s-,\cdot
,y),\varphi \big\rangle q^{(\alpha )}(ds,dy)1_{\alpha \in (0,2)} \\
&&+\int_{0}^{t}\big\langle J^{\beta }h_{n}(s,\cdot ),\varphi \big\rangle %
dW_{s}\,,\quad \varphi \in \mathcal{S}(\mathbf{R}^{d}),
\end{eqnarray*}%
as $n\rightarrow \infty $ and using Lemma \ref{lem2}, we get that the
function $u$ is a strong solution of (\ref{pr3}).

\medskip \emph{Uniqueness}. Let $u\in \mathbb{H}_{p}^{\beta +\alpha }(E)$ be
a strong solution of (\ref{pr3}) with zero input functions $u_{0},f,\Phi ,g$
and $h$. Hence, for every $\varphi \in \mathcal{S}(\mathbf{R}^{d})$ and $%
t\in \lbrack 0,T]$ $\mathbf{P}$-a.s. 
\begin{equation}
\big\langle J^{\beta }u(t,\cdot ),\varphi \big\rangle=\int_{0}^{t}\big\langle%
(A^{(\alpha )}-\lambda )J^{\beta }u(s,\cdot ),\varphi \big\rangle ds
\label{eq25}
\end{equation}

Let $\zeta_{\varepsilon}= \zeta_{\varepsilon}(x)$, $x\in\mathbf{R}^d$, $%
\varepsilon\in(0,1)$, be the functions introduced in Section~3. Inserting $%
\varphi(\cdot)=\zeta_{\varepsilon}(x-{\cdot})$ into (\ref{eq25}), we get
that the function 
\begin{equation*}
\upsilon_{\varepsilon}(t,x)=J^{\beta}u(t,\cdot)\ast\zeta_{\varepsilon}(x)
\end{equation*}%
belongs to $\mathfrak{D}_p(E)$ and 
\begin{equation*}
\upsilon_{\varepsilon}(t,x)=\int_0^t \big(A^{(\alpha)}-\lambda\big) %
\upsilon_{\varepsilon}(s,x)ds.
\end{equation*}%
By Lemma \ref{le16}, $\upsilon_{\varepsilon}=0$ $\mathbf{P}$-a.s. in $E$ for
all $\varepsilon\in(0,1)$. Hence, for every $\varphi\in\mathcal{S}(\mathbf{R}%
^d)$ and $t\in[0,T]$ $\mathbf{P}$-a.s. 
\begin{equation*}
0=\big\langle \upsilon_{\varepsilon}(t,\cdot),\varphi\big\rangle = %
\big\langle J^{\beta}u(t,\cdot)\ast\zeta_{\varepsilon},\varphi\big\rangle %
\to \big\langle J^{\beta}u(t,\cdot),\varphi\big\rangle
\end{equation*}%
as $\varepsilon\to 0$.

The theorem is proved.
\end{proof}

\section{General model}

Finally let us consider the equation (\ref{intr1}). First we solve it for
the smooth input functions. For $g\in \mathfrak{\bar{D}}_{2,p}(E)\cap 
\mathfrak{\bar{D}}_{p,p}(E)$ let%
\begin{equation*}
\Lambda g(t,x,y)=g(t,x-y,y),(t,x)\in E,y\in \mathbf{R}_{0}^{d}.
\end{equation*}%
We define for $\varepsilon >0$%
\begin{equation*}
I_{\varepsilon }g(t,x)=1_{\alpha \in (0,2)}\int_{|y|>\varepsilon }[\Lambda
g(t,x,y)-g(t,x,y)]l^{(\alpha )}(t,y)\frac{dy}{|y|^{d+\alpha }},(t,x)\in E.
\end{equation*}%
If $g,\Lambda g\in \mathfrak{\bar{D}}_{2,p}(E)\cap \mathfrak{\bar{D}}%
_{p,p}(E)$, then for each $\varepsilon >0$ we have $I_{\varepsilon }g\in 
\mathfrak{\bar{D}}_{2,p}(E)\cap \mathfrak{\bar{D}}_{p,p}(E).$

\begin{proposition}
\label{prop2}Let $p\geq 2$ and Assumption A hold. Let 
\begin{eqnarray*}
f &\in &\mathfrak{D}_{p}(E),\Phi \in \mathfrak{D}_{2,p}(E)\cap \mathfrak{D}%
_{p,p}(E), \\
\Lambda g,g &\in &\mathfrak{\bar{D}}_{2,p}(E)\cap \mathfrak{\bar{D}}%
_{p,p}(E),u_{0}\in \mathfrak{D}_{p}(\mathbf{R}^{d})
\end{eqnarray*}%
($u_{0}$ is $\mathcal{F}_{0}$-measurable). Assume that there is $Ig\in 
\mathfrak{D}_{p}(E)$ such that for every $\kappa \in \emph{R}$ and
multiindex $\gamma \in \mathbf{N}_{0}^{d}$ 
\begin{equation*}
\left\vert I_{\varepsilon }g-Ig\right\vert _{\mathbb{H}_{p}^{\kappa }(E)}+%
\mathbf{E}[\sup_{(s,x)\in E}|D_{x}^{\gamma }I_{\varepsilon
}g(s,x)-D_{x}^{\gamma }Ig(s,x)|^{p}]\rightarrow 0
\end{equation*}%
as $\varepsilon \rightarrow 0$ (we denote%
\begin{equation*}
Ig(t,x)=1_{\alpha \in (0,2)}\int [g(t,x-y,y)-g(t,x,y)]l^{(\alpha )}(t,y)%
\frac{dy}{|y|^{d+\alpha }},(t,x)\in E.\text{)}
\end{equation*}

Then there is a unique $u\in \mathfrak{D}_{p}(E)$ solving (\ref{intr1}).
Moreover, $\mathbf{P}$-a.s. $u(t,x)$ is cadlag in $t$ and smooth in $x$, and
there is a constant $C$ independent of $u_{0},f,g,\Phi $ such that%
\begin{eqnarray}
|u|_{\mathbb{H}_{p}^{\beta +\alpha }(E)} &\leq &C[|u_{0}|_{\mathbb{B}%
_{p}^{\beta +\alpha -\frac{\alpha }{p}}(\mathbf{R}^{d})}+|f+Ig|_{\mathbb{H}%
_{p}^{\beta }(E)}+|\Phi |_{\mathbb{H}_{p}^{\beta +\frac{\alpha }{2}%
}(E)}+|\Phi |_{\mathbb{B}_{pp}^{\beta +\alpha -\frac{\alpha }{p}}(E)}
\label{form6} \\
&&+|\Lambda g|_{\mathbb{\bar{H}}_{p}^{\beta +\frac{\alpha }{2}}(E)}+|\Lambda
g|_{\mathbb{\bar{B}}_{pp}^{\beta +\alpha -\frac{\alpha }{p}}(E)}+|h|_{%
\mathbb{H}_{p}^{\beta +\alpha /2}(E,Y)}].  \notag
\end{eqnarray}
\end{proposition}

\begin{proof}
Let%
\begin{eqnarray*}
Y_{t}^{(\alpha )} &=&1_{\alpha \in (0,2)}[\int_{0}^{t}\int \chi ^{(\alpha
)}(y)yq^{(\alpha )}(ds,dy)+\int_{0}^{t}\int (1-\chi ^{(\alpha
)}(y))yp(ds,dy)] \\
&&+1_{\alpha =2}\int_{0}^{t}\sigma (s)dW_{s}.
\end{eqnarray*}%
Consider the problem%
\begin{eqnarray}
dw(t,x) &=&\{\tilde{A}^{(\alpha )}w(t,x)-\lambda w(t,x)+f\left(
t,x-Y_{t}^{(\alpha )}\right) +Ig(t,x-Y_{t}^{(\alpha )})  \label{pr4} \\
&&+\{h(t,x-Y_{t}^{(\alpha )})dW_{t}-1_{\alpha =2}\partial
_{i}h(t,x-Y_{t}^{(\alpha )})\sigma ^{i}(t)dt\}  \notag \\
&&+\int g(t,x-y-Y_{t-}^{(\alpha )},y)q^{(\alpha )}(dt,dy)+\int \Phi
(t,x-Y_{t-}^{(\alpha )},\upsilon )\eta (dt,d\upsilon ),  \notag \\
w(0,x) &=&u_{0}(x),  \notag
\end{eqnarray}%
where $\tilde{A}^{(\alpha )}u$ is defined as $A^{(\alpha )}u$ in (\ref{nf1})
with $m^{(\alpha )}$ replaced by $m^{(\alpha )}-l^{(\alpha )}$ and $%
B^{ij}(t) $ replaced by $B^{ij}(t)-\frac{1}{2}\sigma ^{i}(t)\cdot \sigma
^{j}(t).$ Obviously, 
\begin{eqnarray*}
\Phi (t,x-Y_{t}^{(\alpha )},\upsilon ) &\in &\mathfrak{D}_{2,p}(E)\cap 
\mathfrak{D}_{p,p}(E),f(t,x-Y_{t}^{(\alpha )}), \\
Ig(t,x-Y_{t}^{(\alpha )}),\partial _{i}h^{i}(t,x-Y_{t}^{(\alpha )})\sigma
^{i}(t)1_{\alpha =2} &\in &\mathfrak{D}_{p}(E), \\
g(t,x-y-Y_{t}^{(\alpha )},y) &\in &\mathfrak{\bar{D}}_{2,p}(E)\cap \mathfrak{%
\bar{D}}_{p,p}(E),u_{0}\in \mathfrak{D}_{p}(\mathbf{R}^{d}),
\end{eqnarray*}%
and by Lemma \ref{le16} there is a unique solution $w\in \mathfrak{D}_{p}(E)$
of \emph{(}\ref{pr4}). Moreover, $\mathbf{P}$-a.s. $w(t,x)$ is cadlag in $t$%
, smooth in $x$ and the estimates (\ref{eq20}), (\ref{eq21}) hold. For $%
\varepsilon \in (0,1)$ set%
\begin{eqnarray*}
Y_{t}^{(\alpha ),\varepsilon } &=&1_{\alpha \in
(0,2)}[\int_{0}^{t}\int_{|y|>\varepsilon }\chi ^{(\alpha )}(y)yq^{(\alpha
)}(ds,dy)+\int_{0}^{t}\int (1-\chi ^{(\alpha )}(y))yp(ds,dy)] \\
&&+1_{\alpha =2}\int_{0}^{t}\sigma (s)dW_{s},\bar{Y}_{t}^{(\alpha
),\varepsilon }=1_{\alpha \in (0,2)}\int_{0}^{t}\int_{|y|\leq \varepsilon
}\chi ^{(\alpha )}(y)yq^{(\alpha )}(ds,dy),
\end{eqnarray*}%
$0\leq t\leq T.$ Applying Ito-Wentzel formula (see Proposition 1 of \cite%
{mik1}) we have%
\begin{eqnarray*}
&&w(t,x+Y_{t}^{(\alpha ),\varepsilon }) \\
&=&u_{0}(x)+\int_{0}^{t}\nabla w(s-,x+Y_{s-}^{(\alpha ),\varepsilon
})dY_{s}^{(\alpha ),\varepsilon }+\int_{0}^{t}\int \Phi (s,x,\upsilon )\eta
(ds,d\upsilon ) \\
&&+\sum_{s\leq t}[w\left( s-,x+Y_{s}^{(\alpha ),\varepsilon }\right)
-w(s-,x+Y_{s-}^{(\alpha ),\varepsilon })-\nabla w(s-,x+Y_{s-}^{(\alpha
),\varepsilon })\Delta Y^{(\alpha ),\varepsilon }] \\
&&+\int_{0}^{t}h(s,x)dW_{s}+\int_{0}^{t}\int g(s,x-y,y)q^{(\alpha )}(ds,dy)
\\
&&+\int_{0}^{t}\left( \tilde{A}^{(\alpha )}w(s,x+Y_{s}^{(\alpha
),\varepsilon })-\lambda w(s,x+Y_{s}^{(\alpha ),\varepsilon
})+f(s,x)+Ig(s,x)\right) ds \\
&&+\sum_{s\leq t}\left[ \Delta w(s,x+Y_{s}^{(\alpha ),\varepsilon })-\Delta
w(s,x+Y_{s-}^{(\alpha ),\varepsilon })\right] \\
&&+1_{\alpha =2}\frac{1}{2}\int_{0}^{t}\sigma ^{i}(s)\cdot \sigma
^{j}(s)\partial _{ij}^{2}w(s,x+Y_{s}^{(\alpha ),\varepsilon })ds,0\leq t\leq
T.
\end{eqnarray*}%
Since%
\begin{eqnarray*}
&&\sum_{s\leq t}\left[ \Delta w(s,x+Y_{s}^{(\alpha ),\varepsilon })-\Delta
w(s,x+Y_{s-}^{(\alpha ),\varepsilon })\right] \\
&=&\int_{0}^{t}\int_{|y|>\varepsilon }\left[ g(s,x,y)-g(s,x-y,y)\right]
p^{(\alpha )}(ds,dy) \\
&=&\int_{0}^{t}\int_{|y|>\varepsilon }\left[ g(s,x,y)-g(s,x-y,y)\right]
q^{(\alpha )}(ds,dy)-\int_{0}^{t}I_{\varepsilon }g(s,x)ds,
\end{eqnarray*}%
it follows (by passing to the limit as $\varepsilon \rightarrow 0$) that $%
u(t,x)=w(t,x+Y_{t}^{(\alpha )})$ satisfies (\ref{intr1}). By our assumptions
and Lemma \ref{le16} (the estimate (\ref{eq21}), 
\begin{equation*}
|\partial ^{\gamma }u|_{\mathbb{L}_{p}(E)}<\infty ,\gamma \in \mathbf{N}%
_{0}^{d}
\end{equation*}%
and (\ref{form6}) holds. Therefore $u$ is a solution of (\ref{intr1}). The
uniqueness follows from the fact that we can go backwards. Repeating the
arguments as above we find that if $u\in \mathfrak{D}_{p}(E)$ solves (\ref%
{intr1}) then $w(t,x)=u(t,x-Y_{t}^{(\alpha )})$ is the solution of the class 
$\mathfrak{D}_{p}(E)$ to (\ref{pr4}) for which the uniqueness holds.
\end{proof}

\begin{corollary}
\label{lu3}There is at most one solution $u\in \mathbb{H}_{p}^{\beta +\alpha
}(E)$ of (\ref{intr1}).
\end{corollary}

\begin{proof}
Let $u\in \mathbb{H}_{p}^{\beta +\alpha }(E)$ be a solution to (\ref{intr1})
with zero input functions. Let $\zeta \in C_{0}^{\infty }(\mathbf{R}%
^{d}),\varepsilon >0,\zeta _{\varepsilon }(x)=\varepsilon ^{-d}\zeta
(x/\varepsilon )$ and applying (\ref{for3}) with $\zeta _{\varepsilon
}(x-\cdot )\in C_{0}^{\infty }$ we see that 
\begin{equation*}
u_{\varepsilon }(t,x)=\int u(t,y)\zeta _{\varepsilon }(x-y)dy
\end{equation*}%
belongs to $\mathfrak{D}_{p}(E)$ and solves (\ref{intr1})). Therefore, by
Proposition \ref{prop2} $u_{\varepsilon }(t,x)=0$ for all $\varepsilon >0$.
The statement follows.
\end{proof}

\subsection{Proof of Theorem \protect\ref{main 2}}

By Lemmas \ref{cnew2} and \ref{lemd} there are sequences 
\begin{eqnarray*}
f_{n} &\in &\mathfrak{D}_{p}(E),\Phi _{n}\in \mathfrak{D}_{2,p}(E)\cap 
\mathfrak{D}_{p,p}(E), \\
g_{n} &\in &\mathfrak{\bar{D}}_{2,p}(E)\cap \mathfrak{\bar{D}}%
_{p,p}(E),u_{0,n}\in \mathfrak{D}_{p}(\mathbf{R}^{d})
\end{eqnarray*}%
defined by (\ref{eq14}) such that 
\begin{eqnarray*}
&&|f_{n}-f|_{\mathbb{H}_{p}^{\beta }(E)}+|\Phi _{n}-\Phi |_{\mathbb{H}%
_{2,p}^{\beta +\frac{\alpha }{2}}(E)}+|\Phi _{n}-\Phi |_{\mathbb{B}%
_{p,pp}^{\beta +\alpha -\frac{\alpha }{p}}(E)}+|h_{n}-h|_{\mathbb{H}%
_{p}^{\beta +\alpha /2}(E,Y)} \\
&&+|g_{n}-g|_{\mathbb{\bar{H}}_{2,p}^{\beta +\frac{\alpha }{2}%
}(E)}+|g_{n}-g|_{\mathbb{\bar{B}}_{p,pp}^{\beta +\alpha -\frac{\alpha }{p}%
}(E)}+|u_{0,n}-u_{0}|_{\mathbb{B}_{pp}^{\beta +\alpha -\frac{\alpha }{p}}(%
\mathbf{R}^{d})} \\
&\rightarrow &0\text{ as }n\rightarrow \infty \text{.}
\end{eqnarray*}%
Since $\Lambda g\in \mathbb{\bar{B}}_{p,pp}^{\beta +\alpha -\frac{\alpha }{p}%
}(E)\cap \mathbb{\bar{H}}_{2,p}^{\beta +\frac{\alpha }{2}}(E)$ it follows by
the definition of the approximating sequence that $\Lambda g_{n}\in 
\mathfrak{\bar{D}}_{2,p}(E)\cap \mathfrak{\bar{D}}_{p,p}(E)$ and%
\begin{equation*}
|\Lambda g_{n}-\Lambda g|_{\mathbb{\bar{H}}_{2,p}^{\beta +\frac{\alpha }{2}%
}(E)}+|\Lambda g_{n}-\Lambda g|_{\mathbb{\bar{B}}_{p,pp}^{\beta +\alpha -%
\frac{\alpha }{p}}(E)}\rightarrow 0
\end{equation*}%
as $n\rightarrow \infty $ as well. Since $I_{\varepsilon }g\rightarrow Ig$
in $\mathbb{H}_{p}^{\beta }(E)$ as $\varepsilon \rightarrow 0$ we have for
each $n$ and $\kappa \in \mathbf{R}$ (see estimate of Lemma \ref{cnew2}) 
\begin{equation*}
\left( I_{\varepsilon }g\right) _{n}=I_{\varepsilon }g_{n}\rightarrow \left(
Ig\right) _{n}=Ig_{n}\text{ as }\varepsilon \rightarrow 0\text{ in }\mathbb{H%
}_{p}^{\kappa }(E)\text{,}
\end{equation*}%
where $\left( I_{\varepsilon }g\right) _{n}$ and $\left( Ig\right) _{n}$ are
approximations defined by (\ref{eq14}). In addition, by Lemma \ref{cnew2},%
\begin{equation*}
\left( Ig\right) _{n}=Ig_{n}\rightarrow Ig\text{ in }\mathbb{H}_{p}^{\beta
}(E)\text{ as }n\rightarrow \infty \text{,}
\end{equation*}%
and for each $n$ and multiindex $\gamma \in \mathbf{N}_{0}^{d}$%
\begin{equation*}
\mathbf{E}[\sup_{(s,x)\in E}|D_{x}^{\gamma }I_{\varepsilon
}g_{n}(s,x)-D_{x}^{\gamma }Ig_{n}(s,x)|^{p}]\rightarrow 0
\end{equation*}%
as $\varepsilon \rightarrow 0$. Therefore all the assumptions of Proposition %
\ref{prop2} are satisfied with smooth input functions $%
f_{n},g_{n},h_{n},u_{0,n},\Phi _{n}$. Let us denote $u_{n}$ the
corresponding smooth solution of the class $\mathfrak{D}_{p}(E)$. By
definition,%
\begin{eqnarray}
&&{u}_{n}(t)  \label{f00} \\
&=&u_{n,0}+\int_{0}^{t}[A^{(\alpha )}{u}_{n}{(s)}-\lambda {u}_{n}{(s)}+{f}%
_{n}(s)]ds  \notag \\
&&+\int_{0}^{t}\int_{\mathbf{R}_{0}^{d}}[{u}_{n}(s-,\cdot +y)-{u}%
_{n}(s-,\cdot )+g_{n}(s,\cdot ,y)]q^{(\alpha )}(ds,dy)1_{\alpha \in (0,2)} 
\notag \\
&&+\int_{0}^{t}\int_{U}{\Phi }_{n}(s,\cdot ,\upsilon )\eta (ds,d\upsilon
)+\int_{0}^{t}[1_{\alpha =2}\sigma ^{i}(s)\partial _{i}{u}_{n}(s)+{h}%
_{n}(s)]dW_{s}\,,  \notag \\
0 &\leq &t\leq T.  \notag
\end{eqnarray}%
According to the estimate of Proposition \ref{prop2}, there is a constant $C$
independent of $n,m$ such that%
\begin{eqnarray*}
|u_{n}-u_{m}|_{\mathbb{H}_{p}^{\beta +\alpha }(E)} &\leq
&C[|u_{n,0}-u_{m,0}|_{\mathbb{B}_{p}^{\beta +\alpha -\frac{\alpha }{p}}(%
\mathbf{R}^{d})}+|f_{n}-f_{m}+(Ig_{n}-Ig_{m})|_{\mathbb{H}_{p}^{\beta }(E)}
\\
&&+|\Phi _{n}-\Phi _{m}|_{\mathbb{H}_{p}^{\beta +\frac{\alpha }{2}%
}(E)}+|\Phi _{n}-\Phi _{m}|_{\mathbb{B}_{pp}^{\beta +\alpha -\frac{\alpha }{p%
}}(E)}+|h_{n}-h_{m}|_{\mathbb{H}_{p}^{\beta +\alpha /2}(E,Y)} \\
&&+|\Lambda g_{n}-\Lambda g_{m}|_{\mathbb{\bar{H}}_{p}^{\beta +\frac{\alpha 
}{2}}(E)}+|\Lambda g_{n}-\Lambda g_{m}|_{\mathbb{\bar{B}}_{pp}^{\beta
+\alpha -\frac{\alpha }{p}}(E)}].
\end{eqnarray*}%
Therefore the sequence $u_{n}$ is Cauchy in $\mathbb{H}_{p}^{\beta +\alpha
}(E)$ and there is $u\in \mathbb{H}_{p}^{\beta +\alpha }(E)$ such that $%
|u_{n}-u|_{\mathbb{H}_{p}^{\beta +\alpha }(E)}\rightarrow 0$ as $%
n\rightarrow \infty $. Using Lemmas \ref{ler1}, \ref{lem2}, Corollary \ref%
{corr1} and Theorem 1 in \cite{mir3mol}, we pass easily to the limit in (\ref%
{f00}) as $n\rightarrow \infty $ in $H_{p}^{\beta }(\mathbf{R}^{d})$.
Obviously, $u(t)$ is $H_{p}^{\beta }(\mathbf{R}^{d})$-valued cadlag function.

The uniqueness follows by Corollary \ref{lu3}. Theorem \ref{main 2} is
proved.

\end{document}